\documentclass[11pt,reqno]{amsart}

\usepackage{graphicx}
\usepackage{amssymb}
\usepackage{epsfig}
\usepackage{amsmath}
\usepackage{amsthm}
\usepackage{color}
\definecolor{r}{rgb}{0.9,0.3,0.1}
\definecolor{b}{rgb}{0.1,0.3,0.9}

\usepackage{amsmath}
\usepackage{amssymb}
\usepackage{amsthm}
\usepackage{enumerate}
\usepackage{amsbsy}
\usepackage{amsfonts}
\newtheorem{theo}{Theorem}[section]
\newtheorem{defin}[theo]{Definition}
\newtheorem{prop}[theo]{Proposition}
\newtheorem{coro}[theo]{Corollary}
\newtheorem{lemm}[theo]{Lemma}

\newcommand{\al}{\alpha}
\newcommand{\be}{\beta}
\newcommand{\ga}{\gamma}
\newcommand{\Ga}{\Gamma}

\newcommand{\om}{\omega}
\newcommand{\Om}{\Omega}

\newcommand{\ep}{\epsilon }
\newcommand{\te}{\theta}
\newcommand{\De}{\Delta}
\newcommand{\de}{\delta}

\newcommand{\pa}{\partial}

\newcommand{\R}{{\mathbb R}^n}

\newcommand{\ri}{\rightarrow}
\newcommand{\Rn}{{\mathbb R}^{n-1}}

\newcommand{\na}{\nabla}

\begin{document}
\baselineskip=18pt

\title[]{Boundary value problem of a   non-stationary  Stokes system in a bounded smooth cylinder}

\
\author{Tongkeun Chang}
\address{Department of Mathematics, Yonsei University \\
Seoul, 136-701, South Korea}
\email{chang7357@yonsei.ac.kr}

\author{Bum Ja Jin}
\address{Department of Mathematics, Mokpo National University, Muan-gun 534-729,  South Korea }
\email{bumjajin@hanmail.net}

\thanks{}

\begin{abstract}
In this paper, we intend to study  the boundary value problem of  the  non-stationary  Stokes system
in a bounded smooth cylinder $\Omega\times (0,T)$. As a first step, we  consider the problem in  half-plane cylinder
${\mathbb R}^n_+ \times (0,T), \,\, 0 < T \leq \infty$.
We extend the result of Solonnikov\cite{So} to data in weaker function spaces than the one  considered in \cite{So}.

\end{abstract}

\maketitle

\section{Introduction}
\setcounter{equation}{0}

Let  $\Om$ be a $n$-dimensional bounded $C^2$ domain for $n\geq 3$.
 Let us consider the boundary value problem of a   non-stationary  Stokes system$-$in other words,
 a  non-stationary linearized system of Navier-Stokes equations$-$in a cylindrical domain
 $\Omega\times (0,T)$.
 For a given initial data $f=(f_1,\cdots, f_n)$
 and a given boundary data $g=(g_1,\cdots, g_n)$, find a unknown vector field $u=(u_1,\cdots, u_n)$ and a unknown function $p$
 satisfying the system of equations
\begin{align}\label{maineq}
\begin{array}{ll}\vspace{2mm}
u_t - \De u + \na p =0 & \mbox{ in } \Omega \times (0,T),\\ \vspace{2mm}
div \, u =0  & \mbox{ in } \Omega \times (0,T),\\ \vspace{2mm}
u|_{t=0}= f &  \mbox{ in } \Omega, \\
u|_{\pa \Om \times (0,T)} = g  & \mbox{ on }\partial\Omega \times (0,T).
\end{array}
\end{align}
Here,  $\na p = (\frac{\pa p}{\pa x_1}, \cdots, \frac{\pa
p}{\pa x_n}), \,\, div \, u = \sum_{1 \leq j \leq n} \frac{\pa
u_j}{\pa x_j}$ and $\De$ is the Laplacian.

The  system \eqref{maineq}  has been studied by many mathematicians. Among them,
 Solonnikov\cite{So} showed  that the system \eqref{maineq} has a unique
solution satisfying
\begin{eqnarray}\label{So}
\begin{array}{l}
\| u\|_{W^{2,1 }_{p}(\Om \times (0,T))}\\
 \leq c(T) \Big( \|f\|_{{\mathcal B}_p^{2-\frac{2}{p}}(\Om)}
             + \|g\|_{{\mathcal B }^{2-\frac{1}{p},1-\frac{1}{2p}}_p(\pa \Om  \times (0,T))}
            +  \| g_\nu\|_{{\mathcal B}_p^{2 -\frac1p,1 }( \pa \Om \times (0,T))} \Big).
 \end{array}
\end{eqnarray}
Here, $g_\nu$ denotes the component of $g$ in the direction of the unit outward normal vector $\nu$.
The anisotropic Besov spaces ${\mathcal B}^{\al,\frac12\al }_p (\pa \Om
\times (0,T))$, ${\mathcal B}^{\al,\frac12\al + \frac1{2p} }_p (\pa \Om
\times (0,T))$, and other function spaces are introduced in section \ref{function spaces}.
By making use of the estimates of  the Green matrix in ${\mathbb R}^n_+$, it is shown that the solution  of
the non-stationary Stokes system \eqref{maineq} in  ${\mathbb R}^n_+\times (0,\infty)$
satisfies the estimate
\begin{align*}
&\|D^2_x u\|_{L^p ({\mathbb R}_+ \times (0,\infty)}  + \|D_t  u\|_{L^p ({\mathbb R}_+ \times (0,\infty ))}\\
&\leq c\Big(\|f\|_{{\mathcal B}^{2-\frac{2}{p}}({\mathbb R}^n_+ )}
    +\|g\|_{{\mathcal B}^{2-\frac{1}{p},1-\frac{1}{2p}}_p(\Rn  \times (0,\infty))}
    + \|g_n\|_{{\mathcal B}^{2-\frac{1}{p},1 }_p(\Rn  \times (0,\infty) )} \Big),
\end{align*}
and then, using  flatness near the boundary, the estimate \eqref{So} has been obtained for the solution of the Stokes system
\eqref{maineq} in a smooth bounded cylinder  $\Omega\times (0,T)$.

The aim of  this paper is to extend the result of Solonnikov\cite{So} to data in Besov spaces
$(f,g)\in B^{\alpha -\frac{1}{p}}_p(\Om) \times B^{\al,\frac{\al}{2} }_p(\pa \Om \times (0,T))$, $0<\al< 1 $,
which is weaker   than ${\mathcal B}^{2-\frac{2}{p}}_p(\Omega)\times {\mathcal B}^{2-\frac{1}{p},1-\frac{1}{2p}}_p(\partial \Omega \times (0,T))$.

The system \eqref{maineq} can be decomposed into  the following two systems:
\begin{align}\label{maineq1}
\begin{array}{ll}\vspace{2mm}
v_t - \De v + \na \pi =0 &  \mbox{ in } \Omega \times (0,T),\\ \vspace{2mm}
div \, v =0 & \mbox{ in } \Omega\times (0,T),\\ \vspace{2mm}
v|_{t=0}= f &  \mbox{ in } \Omega, \\
v|_{\pa \Om \times (0,T)}= 0 &  \mbox{ on }\partial\Omega \times (0,T).
\end{array}
\end{align}
and
\begin{align}\label{maineq2}
\begin{array}{ll}\vspace{2mm}
u_t - \De u + \na p =0 &  \mbox{ in } \Omega \times (0,T),\\ \vspace{2mm}
div \, u =0 &  \mbox{ in } \Omega\times (0,T),\\ \vspace{2mm}
u|_{t=0} = 0 &  \mbox{ in } \Omega, \\
u|_{\pa \Om \times (0,T)}= g   & \mbox{ on }\partial\Omega \times (0,T).
\end{array}
\end{align}
The system \eqref{maineq1} has been studied in various function
spaces as a basis of the study of Navier-Stokes system. Indeed,
there are many results on an initial boundary value problem of a
non-stationary Stokes system with no slip boundary
condition (see  \cite{aman}, \cite{giga}, \cite{kato},
\cite{wiegner} and references therein).

In this paper, we consider the solvability of \eqref{maineq2}. The solvability of Stokes system \eqref{maineq2} can lead directly
to the solvability of the original Stokes system \eqref{maineq}.
Our result on the system \eqref{maineq2} could be applied for the study of a  boundary value problem of a non-stationary
Navier-Stokes system with non-zero boundary data. The following is our main result.
\begin{theo}\label{B-1}
Let $1< p<\infty$ and  $0<\alpha<1 $.
Let  $g= (g', g_n) = (g_1,\cdots, g_{n-1}, g_n) \in {\mathcal B}^{\al,\frac12\al }_p ({\mathbb R}^{n-1}
\times (0,T))$ with $g_n \in  {\mathcal B}^{\al,\frac12\al  + \frac1{2p}}_p ({\mathbb R}^{n-1}
\times (0,T))$. Let us also assume $g(x',0) =0$ for $x' \in \Rn$ when $\al > \frac2p$.
Then there is a solution of the Stokes system \eqref{maineq2}, $\Om$ is replaced by ${\mathbb R}^n_+$,  with boundary data $g$ such that
\begin{equation}\label{b-1}
\| u\|_{{\mathcal B}^{\al+\frac{1}{p}, \frac12\al+\frac{1}{2p} }_{p}({\mathbb R}^n_+ \times (0,T))}
 \leq c(T) \Big(\|g\|_{{\mathcal B}^{\al,\frac\al 2}_{p}({\mathbb R}^{n-1} \times (0,T))}
            +  \|g_n\|_{{\mathcal B}_p^{\al,\frac\al2+\frac{1}{2p} }( {\mathbb R}^{n-1} \times (0,T))} \Big).
 \end{equation}
\end{theo}
This result is optimal in the sense that the restrictions  over
${\mathbb R}^{n-1} \times  {\mathbb R}$ of the functions contained
in  ${\mathcal B}^{\al+\frac{1}{p}, \frac12\al +\frac{1}{2p}
}_{p}({\mathbb R}^n \times  {\mathbb R})$
 are in  ${\mathcal B}^{\al, \frac12\al  }_{p}({\mathbb R}^{n-1} \times  {\mathbb R})$ and
 $\| u|_{ {\mathcal B}^{\al, \frac12\al  }_{p}({\mathbb R}^{n-1} \times  {\mathbb R})  }
 \leq c \| u\|_{ {\mathcal B}^{\al+\frac{1}{p}, \frac12\al +\frac{1}{2p}  }_{p}({\mathbb R}^n \times  {\mathbb R})    }      $ (see \cite{C}).

 To obtain   estimates of solutions of system \eqref{maineq2} in a smooth bounded cylinder $\Omega \times (0,T)$, we
 flatten the boundary and make use of   estimates for the solutions of the system \eqref{maineq2} in ${\mathbb R}^n_+\times (0,T)$.
 Then we can extend  Theorem \ref{B-1} to the  Stokes flow in any smooth  bounded cylinder.
\begin{coro}\label{M1}
Let  $\Om$ be a bounded $C^2$-domain in $\R$, $1< p<\infty$, and
$0<\alpha<1 $. Let  $g \in {\mathcal B}^{\al,\frac12\al }_p (\pa \Om
\times (0,T))$ with $ g_\nu \in  {\mathcal B}^{\al,\frac12\al
+ \frac1{2p}}_p (\pa \Om \times (0,T))$ and $\int_{\pa \Om} g_\nu(P,t) dP =0$ for all $0<t<T$. Let us also assume $g(P,0) =0$
for $P \in \pa \Om$ when $\al > \frac2p$.  Then there is a solution
of \eqref{maineq2}  with boundary data $g$ such that
\begin{equation}\label{b-1}
\| u\|_{{\mathcal B}^{\al+\frac{1}{p} , \frac12\al+\frac{1}{2p} }_{p}(\Om \times (0,T))}
 \leq c(T) \Big(\|g\|_{{\mathcal B}^{\al,\frac\al 2}_{p}(\pa \Om \times (0,T))}
            +  \| g_\nu\|_{{\mathcal B}_p^{\al,\frac\al2+\frac{1}{2p} }( \pa \Om \times (0,T))} \Big).
 \end{equation}
\end{coro}

S.  Hofmann,  K. Nystr$\ddot{\rm o}$m\cite{HN}, and Z. Shen\cite{Sh}   have  also considered  Stokes system
\eqref{maineq2} in  ${\mathbb R}^n_+ \times {\mathbb R} $ and
 bounded Lipschitz cylinder  $\Om \times (0,T)$, respectively.
By single layer and double layer potentials of the  Stokes system,  S.  Hofmann and  K. Nystr$\ddot{\rm o}$m\cite{HN} showed
\begin{align*}
 & \int_{ {\mathbb R}^n_+  \times {\mathbb R} } x_n |D_x u|^2   dxdt
\leq c   \| g\|^2_{L^2 ( {\mathbb R}^{n-1} \times {\mathbb R}    )},\\
& \int_{ {\mathbb R}^n_+  \times {\mathbb R} } \Big( x_n |D^2_x u|^2    + x_n |D_t u|^2 \Big)dxdt
\leq c \Big( \| g\|^2_{{\mathcal B}^{1, \frac12}_2 ( {\mathbb R}^{n-1} \times {\mathbb R} )}
  + \int_{{\mathbb R}} \| <\frac{\pa g}{\pa t}, {\bf n}>\|^2_{{\mathcal B}^{-1}_2({\mathbb R}^{n-1} ) }  \Big).
  \end{align*}
 Their results are compared with our results, that is, theorem \ref{Rn} in section \ref{anisotropicsection}.

When $\Om$ is bounded Lipschitz cylinder in ${\mathbb R}^{n+1}$, Z. Shen\cite{Sh}  showed
 \begin{align}\label{Sh}
 \| u^*\|_{L^2(\Om \times (0,T))}
\leq c(T)   \| g\|_{L^2 (\pa \Om \times (0,T))}
 \end{align}
 and
 \begin{align}\label{Sh1}
\notag
\| (D_x u)^*\|_{L^2 ( \pa \Om\times (0,T))} + \| (D^\frac12_t u)^* \|_{L^2( \pa \Om \times (0,T))}
+ \| u^*\|_{L^2( \pa \Om \times (0,T))}\\
\leq c(T) \Big( \| g\|_{{\mathcal B}^{1, \frac12}_2 (\pa \Om \times (0,T))}
  + (\int_0^T\| <\frac{\pa g}{\pa t}, {\bf n}>\|^2_{{\mathcal B}^{-1}_2(\pa \Om) } )^\frac12\Big).
  \end{align}
Here $u^*$ denotes the non-tangential limit of $u$.

We are not sure that the solutions satisfying \eqref{Sh} and
\eqref{Sh1} are in Besov spaces ${\mathcal B}_2^{\frac{1}{2},
\frac{1}{4}}(\Omega\times (0,T))$ and ${\mathcal B}_2^{\frac{3}{2},
\frac{3}{4}}(\Omega\times (0,T))$, respectively, since   for  the
non-stationary Stokes system it is still open problem whether  the
$L^2$ norm of a non-tangential limit of the solution is equivalent
to an area integral of  the solution.

On the other hand, for the solutions of a elliptic equation, a parabolic equation, and the stationary Stokes system it is well known that the $L^2$
norms of a non-tangential limit of the solutions are equivalent to the area integrals of the solutions
  (see \cite{DJK}, \cite{DKPV} on the  elliptic equation, see \cite{B} on the  parabolic  equation,
  and see \cite{BS} for the stationary Stokes system).

We organized  the paper in the following way.
In section \ref{function spaces}, we introduce the anisotropic function spaces.
In section \ref{kernal}, we see that  without loss of generality we can assume the boundary data  $g_n=0$ and then represent the solution of
stokes system in $\R_+ \times (0,\infty)$  by some integral formula.
In section \ref{preliminary} we  study the embedding properties of the functions in weighted Sobolev spaces into  anisotropic spaces and
we introduce on the atomic decomposition of the functions in anisotropic spaces.
In section \ref{pointwise}, we derive pointwise estimates of the solution of the Stokes system \eqref{maineq2} in $\R_+ \times (0,\infty)$
when  the boundary data $g=(g',0)$ is given by an atom.
In section  \ref{anisotropicsection}, we show that the solution of the  Stokes system \eqref{maineq2} in $\R_+ \times (0,\infty)$ is in some
weighted Sobolev spaces in $\R_+ \times (0,\infty)$  using the estimates of section \ref{pointwise} when the boundary data is in the proper
anisotropic space. In section \ref{proofmain} we derive the estimates as in Theorem \ref{B-1} for the boundary data $g=(g',0)$.
Theorem \ref{B-1} for any boundary data $g=(g',g_n)$ will be proved combining  the result of Theorem \ref{B-2} in section  \ref{proofmain}
\and Proposition \ref{prop1} in section \ref{kernal}.

\section{Besov spaces and Anisotropic Besov spaces}

\label{function spaces}
\setcounter{equation}{0}

We denote $x  = (x',x_n) \in {\mathbb R}^n_+$ for $x'\in \Rn $ and denote
$D^{l_0}_{x_n} D^{k_0}_x D^{m_0}_t = \frac{\pa^{l_0} }{\pa x_n} \frac{\pa^{|k_0|}}{\pa x^{k_0}} \frac{\pa^{m_0} }{\pa t}$ for multi index
$l_0, k_0, m_0$. Throughout this paper we denote by $c$ various generic constants and by $c(*,\cdots,*)$
the constants depending only on the quantities appearing in the subindex.

Let $\Om $ be ${\mathbb R}^n_+ $  or a bounded domain.
 For   $1 \leq p \leq \infty$ and $0 < \al $,  the Besov space ${\mathcal B}^\al_p(\Om)$   is  set of functions  satisfying
\begin{align*}
\| f\|^p_{{\mathcal B}^\al_p(\Om)} = \| f\|_{W_p^{[\al]}(\Om)}^p + \sum_{|k |= [\al]} \int_{\Om}\int_{\Om}
\frac{|D^k_x f(x) - D^k_y f(y)|^p}{|x-y|^{n +p(\al - |k|)}} dxdy < \infty, \quad 1 \leq p < \infty,\\
\| f\|_{{\mathcal B}^\al_\infty(\Om)} = \sup_{x \in \Om} |D^{[\al]}_x u(x)| + \sup_{|k| =[\al]}
\sum_{x,y \in \Om}\frac{|D^k_xu(x) - D^k_y u(y)|}{|x-y|^{\al -|k|}} < \infty, \quad p =\infty.
\end{align*}
Here $[\al]$ denotes the largest integer less than $\al$ and $W_p^{[\al] }(\Om)$ is the usual Sobolev space in $\Om$.

For interval  such as $I =(0,T)$,  $(0, \infty)$  or ${\mathbb R}$,  $ {\mathcal B}^\al_p(I )$ is  defined similarly for $1\leq p\leq \infty$ and $\al>0$.

For $0 < \al <2$ and $0<\be < 2$, we define the  anisotropic Besov spaces  ${\mathcal B}^{\al,\be}_p (\Om \times I
)$ by the Banach spaces
\begin{align}\label{p-besov}
{\mathcal
B}^{\al,\be}_p (\Om \times I  ) = L^p(I; {\mathcal
B}^{\al }_p (\Om)) \cap L^p(\Om; {\mathcal
B}^{\be }_p (I))
\end{align}
with norm
\begin{align*}
\| u\|^p_{{\mathcal B}^{\al,\be}_p
(\Om \times I )} &:= \int_{I } \|
u(\cdot, t)\|_{{\mathcal B}^\al_p (\Om)}^p dt
+ \int_{\Om} \| u(x,\cdot)\|^p_{{\mathcal B}^{\be}_p (I )} dx\mbox{ for } 1 \leq p < \infty,\\
\| u\|_{{\mathcal B}^{\al,\be}_\infty
(\Om \times I   )} & := \sup_{t
\in I } \| u(\cdot, t)\|_{{\mathcal B}^\al_\infty (\Om)} +
\sup_{ x \in \Om} \| u(x,\cdot)\|_{{\mathcal
B}^{\be}_\infty (I )}.
\end{align*}
It is well known theory that the usual Besov spaces are real
interpolation spaces of Sobolev spaces:
\begin{align*}
{\mathcal B}^\al_p(\Om) &= \left\{\begin{array}{ll} \vspace{2mm}
(W^1_p (\Om ), L^p(\Om))_{1-\al,p}&\mbox{if } 0 < \al <1, \\
(W^2_p(\Om), W^1_p (\Om))_{2-\al,p}& \mbox{if } 1 < \al <2.
\end{array}
\right.\\
{\mathcal B}^{\frac12\al}_p(I ) & =
(W^1_p (I ), L^p(I ))_{1-\frac12 \al,p}\mbox{ if } 0 < \al <2,
\end{align*}
for $1\leq p\leq \infty$ (see Proposition 2.17 in \cite{JK}). It is
also well-known that for $1 \leq p < \infty$
\begin{align}\label{interpolation}
\begin{array}{ll}
L^p(I; {\mathcal B}^{\al}_p(\Om)) &= \left\{\begin{array}{ll} \vspace{2mm}
(L^p(I; W^{1}_p (\Om)),L^p(I; L^p (\Om)))_{1-\al,p} & \mbox{if} \quad 0 < \al < 1,\\
(L^p(I; W^{2}_p (\Om)),L^p(I; w^1_p (\Om)))_{2-\al,p} & \mbox{if} \quad 1 < \al < 2
\end{array}
\right.
\\
L^p(\Om; {\mathcal B}^{\frac12 \al}_p(I)) &= (L^p(\Om; W^{1}_p (I)),L^p(\Om; L^p (I )))_{1-\frac12 \al,p}
\end{array}
\end{align}
(see Comment 5.8.6 in  \cite{BL}).

\section{Solution formula of the Stokes system in ${\mathbb R}^n_+ \times (0,\infty)$}
\label{kernal}
\setcounter{equation}{0}

Let $g \in C^\infty_c(\Rn \times(0,\infty))$.
We decompose $g$ by $g = g^1 + g^2 =  (R g_n, g_n) + (g' - R g_n, 0 ) $, where $R = (R_1, R_2, \cdots , R_{n-1})$ is Riesz transform. Let
$$
\phi(x,t)=-\om_n \int_{{\mathbb R}^{n-1} } E(x'-y',x_n)g_n(y',t)dy',
$$
where $E$ is a fundamental solution of Laplace equation.
Then, $(\na \phi, -\phi_t)$ satisfies the Stokes system \eqref{maineq2} for $\Om = {\mathbb R}^n_+$ with boundary data $g^1$.
The following estimate is  well  known property of the singular  integral operator (see \cite{St}).
\begin{prop}
\label{prop1}
Let $1 < p < \infty$ and $\al >0$. Then there is a positive constant $c$ such that
\begin{align*}
\| \nabla\phi\|_{{\mathcal B}^{\al+\frac{1}{p}, \frac\al 2+\frac{1}{2p} }_p(\R_+\times (0, T))}
\leq c(T)\| g_n\|_{{\mathcal B}_p^{\al, \frac\al2+\frac{1}{2p}}(\Rn \times (0, T))}
\end{align*}
for $0 < T < \infty$.
\end{prop}
Let $(u,p)$ be the solution of the Stokes system \eqref{maineq2} in $\Om = {\mathbb R}^n_+$  with boundary data $g^2$. Then, $(\na \phi +u, -\phi_t +p)$ satisfies
the Stokes system \eqref{maineq2}  with boundary data $g$.
 Hence, to prove theorem \ref{B-1} we have only to consider the Stokes system \eqref{maineq2} with the boundary data $g^2$. Note that $g^2_n =0$.

From now on, without loss of generality, we assume  $g_n=0$.
  The solution $(u, p)$ of the Stokes system \eqref{maineq2} with boundary data $g$ with $g_n=0$ is represented by
\begin{align}\label{representation}
\begin{array}{ll}\vspace{2mm}
u^i(x,t) &= \sum_{j=1}^{n-1}\int_0^t \int_{\Rn} K_{i,j}( x'-y',x_n,t-s)g_j(y',s) dy'ds,\\
p(x,t) &= \sum_{j=1}^{n-1}\int_0^t \int_{\Rn} \pi_j(x'-y',x_n,t-s) g_j(y',s) dy'ds,
\end{array}
\end{align}
where
\begin{align*}
K_{ij}(x,t)
 &=-2 \delta_{ij}D_{x_n}  \Ga(x,t)
 +4G_{ij} (x,t)
\end{align*}
and
\begin{align*}
\pi_j (x,t)&=-2\delta(t)\frac{\partial^2}{\partial x_j\partial x_n}E(x)+4\frac{\partial^2}{\partial x_n^2}A(x,t)
+4\frac{\partial}{\partial t}\frac{\partial}{\partial x_j}A(x,t),
\end{align*}
where $\Ga$ and $E$ are  the fundamental solutions of the heat equation and Laplace equation, respectively, and
\begin{align*}
 {G}_{ij} (x,t) & =D_{x_j}\int_0^{x_n} \int_{\Rn}  D_{z_n}  \Ga(z,t)  D_{x_i} E(x-z)  dz,\\
 A(x,t)&=\int_{\Rn}\Ga(z',0,t)E(x'-z',x_n)dz'.
\end{align*}

 $G_{ij}$ and $A$ satisfy the estimates
\begin{align}\label{inequality11}
|D^{l_0}_{x_n} D^{k_0}_{x'} D_{t}^{m_0} G_{ij}(x,t)|
& \leq  \frac{c}{t^{m_0 + \frac12} (|x|^2 +t )^{\frac12 n + \frac12 k_0} (x_n^2 +t)^{\frac12 l_0}},\\
|D^j_xD^m_tA(x,t)|&\leq \frac{c}{t^{m+\frac{1}{2}}(|x|^2+t)^{\frac{n-2+|j|}{2}}},
\end{align}
where $ 1 \leq  i \leq n$ and $1 \leq j \leq n-1$ (see \cite{K} and \cite{So}).
The estimates \eqref{inequality11} of $G_{ij}$ and the estimate of Gaussian kernel $\Gamma$ imply that
\begin{align}\label{inequality1}
|D^{l_0}_{x_n} D^{k_0}_{x'} D_{t}^{m_0} K_{ij}(x,t)|
& \leq  \frac{c}{t^{m_0 + \frac12} (|x|^2 +t )^{\frac12 n + \frac12 k_0} (x_n^2 +t)^{\frac12 l_0}}.
\end{align}

\section{Preliminary theories}
\setcounter{equation}{0}
\label{preliminary}
\subsection{Estimates weighted Sobolev spaces in  Anisotropic spaces}
\label{weighted}

\begin{lemm}\label{trace}
Let $0< \al < 1$ and  $1 \leq p < \infty$. Let $u\in C^\infty({\mathbb R}^n_+\times (0,\infty))$.
Then
\begin{align}\label{interpolation2}
\begin{array}{ll}\vspace{2mm}
 \| u\|^p_{{\mathcal B}_p^{\al,\frac12 \al} (\R_+  \times (0,\infty) )}
 & \leq c\int  \int_{\R_+ \times (0,\infty)}
    (x_n \wedge t^\frac12)^{p-p\al} \Big( |D_{x}  u  |^p    + |u   |^p\Big) \\
    &\qquad +   (x_n \wedge t^\frac12)^{2p-p\al}\Big(|u|^p+| D_t u  |^p\Big) d x dt.
\end{array}
\end{align}
Here, $a\wedge b=\min\{a,b\}.$
\end{lemm}
\begin{proof}
By the property of real interpolation $\eqref{interpolation}_1$, we note that
\begin{eqnarray*}
\begin{array}{l}
\| u \|_{L^p((0,\infty):{\mathcal B}_p^{\al}   ( \R_+ ))}\\
 \leq  \inf \Big[\Big(\int^\infty_0 \|s^{1- \al} f(s)\|^p_{ L^p((0,\infty) ; W^1_p   (\R_+ ))} s^{-1} ds \Big)^{\frac1p}
 + \Big(\int^\infty_0 \|s^{1- \al} f'(s)\|^p_{L^p((0,\infty) ;L^p   (\R_+ ))} s^{-1} ds \Big)^{\frac1p}\Big],
\end{array}
\end{eqnarray*}
where infimum is taken by $f:[0,\infty)
\ri L^p((0,\infty) ; W^1_p   (\R_+ )) + L^p((0, \infty) ; L^p (\R_+ ))$  satisfying
$f(0) =  u $ (see Theorem 3.12.2 in \cite{BL}).  Define $f(s) = u (x',x_n + s ,t+
s^2) $. Then $f(0) =  u$, and hence
\begin{align*}
  \|  u \|^p_{L^p((0,\infty) ;{\mathcal B}_p^{\al}   (\R_+))}
& \leq \int_0^\infty \int^\infty_0 \|s^{1- \al} f(s)\|^p_{  W^1_p   (\R_+)} s^{-1} dtds
+ \int_0^\infty \int^\infty_0 \|s^{1- \al} f'(s)\|^p_{ L^p (\R_+ )} s^{-1} dtds\\
& \leq  c \int_0^\infty \int_0^\infty  \int_{\R_+}s^{p-p\al -1}\Big(|u (x',x_n +
          s,t+s^2)|^p + |D_x u(x', x_n + s,t+s^2)|^p\\
& \qquad  + s^p|D_t u (x',x_n +s ,t+ s^2)|^p \Big)  dx dtds.
\end{align*}
Changing variables and exchanging the order of integrations,  the right hand side of the above last inequality is less than
\begin{eqnarray*}
\int_0^\infty \int_{\R_+} \Big(|u   (x,t)|^p+ |D_x u (x,t)|^p \Big)  \int_0^{ x_n \wedge t^\frac12} s^{p - \al p -1} ds
+ |D_tu   (x,t)|^p \Big(\int_0^{x_n \wedge t^\frac12} s^{2p - \al p -1} ds \Big) dxdt.
\end{eqnarray*}
Since $p-p\al >0$, the above is dominated by
\begin{eqnarray*}
 \int_0^\infty \int_{\R_+} (x_n \wedge t^\frac12)^{ p-p\al}\Big(
|D_x u (x,t)|^p + |u  (x,t)|^p + (x_n \wedge t^\frac12)^p|D_t u (x,t)|\Big)  dxdt.
\end{eqnarray*}

Next, we define $g(s) =  u (x',x_n + s^\frac12 ,t+
s)$. Then, $g(0)=u$ and by the property of real interpolation  $\eqref{interpolation}_2$,  we have
\begin{align*}
  \| u\|^p_{L^p(\R_+; {\mathcal B}_p^{\frac12 \al}   (0,\infty))}
&\leq  \int^\infty_0 \int_{\R_+} \Big( \|s^{1- \frac{\al}{2}} g(s)\|^p_{W^1_p   (0,\infty)} s^{-1} ds
+  \|s^{1- \frac{\al}{2}} g'(s)\|^p_{L^p (0,\infty)} s^{-1} \Big) ds\\
&\leq \int^\infty_0\int^\infty_0 \int_{\R_+} s^{p- \frac{\al p}{2}-1} \Big(| u(x',x_n+s, t+s^2)|^p+ | D_tu(x',x_n+s, t+s^2)|^p\Big)   \\
&\quad+s^{\frac{p}{2}- \frac{\al p}{2}-1}|D_{x_n}u(x',x_n+s, t+s^2)|^p dxdt ds.
\end{align*}
Changing variables and exchanging the order of integrations,  the right hand side of the above last inequality is less than
\begin{align*}
&\int_0^\infty \int_{\R_+} \Big(|u   (x,t)|^p+ |D_t u (x,t)|^p \Big)  \int_0^{ x_n^2 \wedge t} s^{p - \frac{\al p}{2} -1} ds
+ |D_{x_n}u   (x,t)|^p \Big(\int_0^{x_n^2 \wedge t} s^{\frac{p}{2} - \frac{\al p}{2} -1} ds \Big) dxdt\\
 &\leq c \int_0^\infty \int_{\R_+} (x_n \wedge t^\frac12)^{ 2p-p\al}\Big(
|D_t u (x,t)|^p + |u  (x,t)|^p+ (x_n \wedge t^\frac12)^{-p}|D_{x_n} u (x,t)|\Big)  dxdt.
\end{align*}
\end{proof}

\begin{lemm}\label{trace2}
Let $1< \al < 2$ and  $1 \leq p < \infty$. Let $u\in  C^\infty({\mathbb R}^n_+\times (0,\infty)).$
Then
\begin{align}\label{al-2}
 \notag \| u\|^p_{{\mathcal B}_p^{\al,\frac12 \al} (\R_+  \times (0,\infty) )}
 & \leq c \int \int_{\R_+ \times (0,\infty)}
  (x_n \wedge t^\frac12)^{2p-p\al}  \Big( |D^2_{x}  u  |^p + |D_x u|^p + |u   |^p+|D_tu|^p\Big) \\
  &\qquad
  +   (x_n \wedge t^\frac12)^{3p-p\al}\Big( |D_tu|^p+|D_x D_t u  |^p+|D_x u|^p\Big)
  d x dt.
\end{align}
\end{lemm}

\begin{proof}
As in the proof of  Lemma \ref{trace}, we define $f(s) =   u (x,x_n + s ,t+
s^2) $. Then $f(0) = u $, and by the property of real interpolation $\eqref{interpolation}_1$, we have
\begin{align*}
 &\| u\|^p_{L^p( (0,\infty); {\mathcal B}^\al_p(\R_+))}\\
  & \leq \int_0^\infty\int^\infty_0 \|s^{2- \al} f(s)\|^p_{ W^2_p   (\R_+ )} s^{-1} dtds
+ \int_0^\infty \int^\infty_0 \|s^{2- \al} f'(s)\|^p_{W^1_p (\R_+ )} s^{-1} dtds\\
& \leq \int_0^\infty \int_0^\infty s^{2p-p\al -1}  \int_{\R_+}|u (x',x_n +
s,t+s^2)|^p + |D_x u(x', x_n + s,t+s^2)|^p  \\
&\qquad\qquad\qquad\qquad\qquad+  |D^2_x u(x', x_n + s,t+s^2)|^p\\
&\qquad+ s^p \Big(|D_t u (x',x_n +s ,t+ s^2)|^p+ |D_x D_t u (x',x_n +s,t+ s^2)|^p   \Big) dx dtds.
\end{align*}
By changing  variables and exchanging the order of integration, the right hand side of the above last inequality is less than
\begin{align*}
 &c \int \int_{\R_+ \times (0,\infty)}
    \Big( |D^2_{x}  u  |^p +|D_x u|^p  + |u   |^p \Big)\int_0^{x_n\wedge t^{\frac{1}{2}}}s^{2p-p\al-1}ds\\
  &\qquad \qquad\qquad  + \Big(|D_tu|^p
  +   |D_x D_t u  |^p\Big) \big(\int_0^{x_n\wedge t^{\frac{1}{2}}}s^{3p-p\al-1}ds \Big)d x dt\\
&\leq  c \int \int_{\R_+ \times (0,\infty)}
  (x_n \wedge t^\frac12)^{2p-p\al}  \Big( |D^2_{x}  u  |^p +|D_x u|^p  + |u   |^p \Big)\\
  &\qquad\qquad+ (x_n \wedge t^\frac12)^{3p-p\al}\Big(|D_tu|^p
  +   |D_x D_t u  |^p\Big) d x dt.
  \end{align*}

Next, we define $g(s) =  u (x',x_n + s^\frac12 ,t+
s)-    s^\frac12 D_{x_n} u  (x',x_n +
s^\frac12 ,t+ s)$. Then by the property of real interpolation $\eqref{interpolation}_2$, we have
\begin{align}
\label{z}
 \|  u\|^p_{ L^p(\R_+;{\mathcal B}_p^{\frac12 \al}   (0,\infty))}
\leq  \int^\infty_0 \int_{\R_+}  \Big( \|s^{1- \frac{\al}{2}} g(s)\|^p_{ W^1_p  (0,\infty)} s^{-1}
+   \int^\infty_0 \|s^{1- \frac{\al}{2}} g'(s)\|^p_{L^p (0,\infty)} s^{-1} \Big) dxds.
\end{align}
The first integration term of the right-hand side of \eqref{z} is dominated by
\begin{align*}
& \int_0^\infty \int_0^\infty \int_{\R_+}  s^{p(1-\frac{\al}2)-1}
\Big( |D_t u(x',x_n + s^\frac12 ,t+ s)|^p+|u(x',x_n + s^\frac12 ,t+ s)|^p \Big)\\
&\quad+ s^{\frac{3p}{2}-\frac{\al p}{2}-1} \Big(|D_{x_n} D_t u(x',x_n + s^\frac12 ,t+ s)|^p+|D_{x_n} u(x',x_n + s^\frac12 ,t+ s)|^p   \Big) dxdtds
\\
& \leq c \int_0^\infty \int_{\R_+} (x_n \wedge t^\frac12)^{ p(2-\al)}\Big(
| D_tu (x,t)|^p+|u(x,t)|^p\Big)\\
&\qquad\qquad +  (x_n \wedge t^\frac12)^{(3-\al) p}\Big( |D_{x_n}u(x,t)|^p+|D_{x_n}D_t u (x,t)|^p\Big)  dxdt.
\end{align*}

To estimate the second integration term  on the right-hand side of \eqref{z}, we note that $g'(s)
=   D_t u   + \frac12D^2_{x_n} u   - s^{\frac12} D_{x_n} D_t u $. By the same reasoning as for the estimate of the first term, the
second term is dominated by
\begin{align*}
\int_0^\infty \int_{\R_+} (x_n \wedge t^\frac12)^{ p(2-\al)}\Big(
| D_t (x,t)|^p+|D_x^2 u(x,t)|^p\Big) + (x_n \wedge t^\frac12)^{(3-\al)p}|D_x D_t u (x,t)|dxdt.
\end{align*}
\end{proof}

\subsection{ Atom decomposition of the functions in the anisotropic Besov spaces}
\label{atom decompositionsection}

Now, we introduce atomic decomposition of functions in anisotropic space (see  \cite{Bo} for the reference and
see also \cite{JK} for atomic decomposition  of functions in  Besov spaces).

\begin{defin}[Definition 5.2 in \cite{Bo}]

Let $\al > 0$ and $1 \leq p \leq \infty$.
An $(\al, p)$-atom is a function in $\Rn \times
 {\mathbb R} $ satisfying
\begin{align}\label{atom}
 |a| \leq r^{\al -\frac{n+1}p}, \,\, |D_{x'} a|
\leq r^{\al -\frac{n+1}p -1 },\,\,|D_t a| \leq r^{\al -\frac{n+1}p -2}, \, \,\, supp\,\, a
\subset \De(y_0',r) \times (t_0 , t_0 + r^2)
\end{align}
for some $r>0$ and $(y_0',t_0) \in \Rn
\times {\mathbb R} $, where $ \De (y_0',r) = \{ y' \in \Rn  \,
| \, |y_0' -y'| < r \}$.
\end{defin}

\begin{prop}[Theorem 5.10 in \cite{Bo}]\label{atom decomposition}
Let $\al > 0$ and $1 \leq p \leq \infty$ and  $g \in {\mathcal B}^{\al,\frac12\al}_p (\Rn \times
{\mathbb R} )$. Then there are sequences $\{a_k\}_{1 \leq k < \infty} $ and $ \{ c_k\}_{1\leq k < \infty}$ of  atoms
and real numbers such that $g = \sum c_k a_k $ and
\begin{align*}
 \big(\sum_{1 \leq k < \infty} |c_k|^p \big)^\frac1p \leq c\| g \|_{{\mathcal B}^{\al, \frac12\al}_p (\Rn \times {\mathbb R}) }.
\end{align*}
\end{prop}

\section{Pointwise estimates of Solution in $\R_+ \times (0,\infty)$}\label{pointwise}
\setcounter{equation}{0}
In this section, we would like to derive pointwise estimates of solution of the Stokes system \eqref{maineq2} with boundary data
$g=(g',0)\in {\mathcal B}^{\al, \frac12\al}_p (\Rn \times {\mathbb R})$, for $1 \leq p \leq  \infty$
and $0 <\al < 1$ (by the reasoning in section \ref{kernal}, we may assume $g_n =0$).
From Proposition \ref{atom decomposition}, without loss of generality, we assume that the component functions $g'_k, \,\, 1 \leq k \leq n-1$ of
$g=(g',0)$ consist of  $(\alpha,p)$-atoms. For simplicity, assume $g'_k=a\delta_{kj}$ for fixed
$1 \leq j \leq n-1$, where  $a$ is an  $(\al, p)$- atom such that $supp \,\, a \subset \De(0,r) \times (0, r^2)$. By \eqref{representation},
the solution $u=(u_1,\cdots, u_n)$ of \eqref{maineq2} is represented by
\begin{align}
\label{sample1}
u^i(x,t) = \int_0^t \int_{\Rn} K_{ij}(x'-y', x_n, t-s) a(y', s) dy'ds, \quad 1 \leq i \leq n.
\end{align}

\begin{lemm}\label{Lemma-0}
Let $t \geq (2r)^2$. Then
\begin{align}\label{EQ1}
&\notag|D_{x_n}^{l_0} D_{x'}^{k_0} D^{m_0}_t  u(x,t)| \\
& \leq c\left\{ \begin{array}{ll} \vspace{2mm}
r^{\al -\frac{n+1}p +n+1} t^{-\frac12 -m_0} (  |x'|^2 + x_n^2 + t)^{-\frac{n+k_0}2} (x_n^2 +t)^{-\frac{l_0}2} & \mbox{if } |x'| \geq 2r,\\
r^{\al -\frac{n+1}p + n+1} t^{-\frac12 -m_0} (  x_n^2  + t )^{-\frac{n+k_0 + l_0}2}& \mbox{if } |x'| \leq 2r.
\end{array}
\right.
\end{align}

\end{lemm}

\begin{proof}
Note that  for $1 \leq i \leq n$,  we have
\begin{align}\label{form}
D_{x_n}^{l_0} D_{x'}^{k_0} D^{m_0}_{t}  u^i(x,t)
= \int_0^{r^2} \int_{\Rn}   D_{x_n}^{l_0} D_{x'}^{k_0} D^{m_0}_t K_{ij} (x' -y', x_n, t-s)  a (y',s) dy'ds.
\end{align}
Since $t \geq (2r)^2$, from  the estimate \eqref{inequality1} of $K_{ij}$ and \eqref{atom}, we have
\begin{align}\label{EQ1-1}
&\notag|D_{x_n}^{l_0} D_{x'}^{k_0} D^{m_0}_t  u ^i(x,t)| \\
&\notag\leq c r^{\al -\frac{n+1}p } \int_0^{r^2}\int_{|y'|< r}
(t-s)^{-\frac12 -m_0} (|x'-y'|^2 + x_n^2+ t-s)^{-\frac{n}2 -\frac{k_0}2}  (x_n^2 + t-s)^{-\frac{l_0}2}dy'ds\\
& \leq c r^{\al -\frac{n+1}p } \int_0^{r^2}\int_{|y'|< r}
t^{-\frac12 -m_0} (|x'-y'|^2 + x_n^2+ t)^{-\frac{n}2 -\frac{k_0}2}  (x_n^2 + t)^{-\frac{l_0}2}dy'ds.
\end{align}
Note that for $\te_1> \frac{n-1}2$, we have
\begin{align}\label{0824}
\int_{|y'| < r} (|x'-y'|^2 + x_n^2+ t)^{-\te_1 }  dy'
 \leq c\left\{ \begin{array}{ll} \vspace{2mm}
  r^{n-1}(  |x'|^2 +  x_n^2 +t)^{-\te_1} ,& \mbox{if } |x'| \geq 2r\\
r^{n-1} (  x_n^2  +t )^{-\te_1 },& \mbox{if } |x'| \leq 2r.
\end{array}\right.
\end{align}
Taking $\te_1 = \frac{n+ k_0}2 > \frac{n-1}2$ in \eqref{0824} and applying to the right hand side of \eqref{EQ1-1},
we obtain the estimate \eqref{EQ1}.
\end{proof}
\begin{lemm}\label{Lemma-1}
Let $t \leq (2r)^2$.
\begin{enumerate}
\item
If $|x'|\leq 2r$ and $ x_n^2\leq t$, then
\begin{align}
\label{integral1}
|D_{x_n}^{l_0} D_{x'}^{k_0} D^{m_0}_t  u (x,t)|
 \leq
\left\{ \begin{array}{ll} \vspace{2mm}
c r^{\al -\frac{n+1}p-2m_0 }x_n^{-( k_0 + l_0)} & \mbox{ for }k_0+l_0\geq 1,\\
c r^{\al -\frac{n+1}p-2m_0 }   \ln ( 1 + { \frac{t}{x_n^2} } )   & \mbox{ for }k_0+l_0=0.
\end{array}\right.
\end{align}

\item
If $|x'|\leq 2r$ and $x_n^2\geq t$, then
\begin{align}
\label{integral2}
 |D_{x_n}^{l_0} D_{x'}^{k_0} D^{m_0}_t  u(x,t)|
 \leq c r^{\al -\frac{n+1}p-2m_0 }x_n^{-(k_0 + l_0+1) }t^{\frac{1}{2}}.
 \end{align}

\item
If $|x'|\geq 2r$ and $ x_n^2\leq t$, then
\begin{align}
\label{integral3}
 \notag&|D_{x_n}^{l_0} D_{x'}^{k_0} D^{m_0}_t  u (x,t)|\\
 &\leq \left\{\begin{array}{ll} \vspace{2mm}
 cr^{\al -\frac{n+1}p-2m_0+n-1 } (|x'|^2+x_n^2)^{\frac{-n-k_0}2} \ln ( 1 + {\frac{t}{x_n^2}} )&\mbox{ for }l_0=1,\\
cr^{\al -\frac{n+1}p-2m_0+n-1 } (|x'|^2+x_n^2)^{ \frac{-n-k_0}2} (x_n^2 + t)^{ \frac{1-l_0}2} &\mbox{ for }l_0\neq 1.
 \end{array}\right.
\end{align}

\item
If $|x'|\geq 2r$ and $ x_n^2\geq t$, then
\begin{align}
\label{integral4}
 |D_{x_n}^{l_0} D_{x'}^{k_0} D^{m_0}_t  u(x,t)|
 \leq c r^{\al -\frac{n+1}p-2m_0+n-1 } x_n^{-l_0}(|x'|^2+x_n^2)^{-\frac{n+k_0}{2}}t^{\frac{1}{2}}.
\end{align}
\end{enumerate}
\end{lemm}
\begin{proof}
By the uniqueness of the  solution of the initial-boundary value
problem, we have $D_t u(x,t) = \int_0^t \int_{\Rn}    K(x'-y', x_n,
t-s) D_s a(y',s) dy'ds$. This implies
\begin{align}\label{form}
D_{x_n}^{l_0} D_{x'}^{k_0} D^{m_0}_t u^i(x,t)
= \int_0^t \int_{\Rn}   D_{x_n}^{l_0} D_{x'}^{k_0} K_{ij} (x' -y', x_n, t-s) D^{m_0}_s a (y',s) dy'ds.
\end{align}
Since $t\leq (2r)^2$, from the estimate \eqref{inequality1} of $K_{ij}$ and \eqref{atom}, and taking
$\te_1: = \frac{n+ k_0}2 > \frac{n-1}2$ in  \eqref{0824}, we have
\begin{align}
\label{integral-1}
& \notag|D_{x_n}^{l_0} D_{x'}^{k_0} D^{m_0}_t  u ^i(x,t)| \\
\notag&\leq c r^{\al -\frac{n+1}p-2m_0 } \int_0^{t}\int_{|y'|< r}
 (t-s)^{-\frac12} (|x'-y'|^2 + x_n^2+ t-s)^{\frac{ -n- k_0}2}  (x_n^2 + t-s)^{-\frac{l_0}2}dy'ds\\
 \notag&\leq c r^{\al -\frac{n+1}p-2m_0 } \int_0^{t}\int_{|y'|<r}
 s^{-\frac12} (|x'-y'|^2 + x_n^2+ s)^{\frac{ -n- k_0}2} (x_n^2 + s)^{-\frac{l_0}2}dy'ds\\
 & \leq
 \left\{\begin{array}{ll} \vspace{2mm}
 c r^{\al -\frac{n+1}p-2m_0+n-1}  (  |x'|^2 + x_n^2)^{\frac{ -n- k_0}2} \int_0^{t}s^{-\frac12} (x_n^2 + s)^{-\frac{l_0}2} ds
                     &\mbox{if }|x'|\geq 2r ,\\
 c r^{\al -\frac{n+1}p-2m_0 } \int_0^{t}s^{-\frac12} (x_n^2 + s)^{\frac{ -n- k_0 -1}2} ds &\mbox{if }|x'|\leq 2r.
 \end{array}\right.
\end{align}
Note that for $\te_2\geq 0$,  if $x_n^2\geq t$, then
\begin{align}\label{0822}
\int_0^{t}s^{-\frac12} ( x_n^2 + s)^{- \te_2 } ds
\leq c x_n^{-2 \te_2 }\int_0^{t}s^{-\frac12} ds
\leq c x_n^{-2\te_2} t^\frac12,
\end{align}
and if $x_n^2\leq t$, then
\begin{align} \label{0822-1}
\notag \int_0^{t}s^{-\frac12} (x_n^2 + s)^{-\te_2 } ds
&\leq x_n^{-2\te_2 }\int_0^{x_n^2}s^{-\frac12}ds+\int_{x_n^2}^t s^{-\frac12 - \te_2 } ds\\
&\leq \left\{\begin{array}{ll} \vspace{2mm}
ct^{-\te_2 +\frac12} &\mbox{ for }  \te_2 < \frac12, \\\vspace{2mm}
cx_n^{-2\te_2 +1} &\mbox{ for } \te_2 > \frac12, \\
 c \ln( 1 + {\frac{t}{x_n^2}}) &\mbox{ for } \te_2 = \frac12.
\end{array}\right.
\end{align}
Taking $\te_2: = \frac{k_0+ l_0 +1}2 \geq \frac12 $ in \eqref{0822} - \eqref{0822-1} and applying to the right hand side of
 $\eqref{integral-1}_2$, we  obtain the estimates $\eqref{integral1}$  and \eqref{integral2}.
And taking $\te_2: = \frac{l_0}2 \geq 0 $ in \eqref{0822} -
\eqref{0822-1} and applying to the right hand side of
 $\eqref{integral-1}_1$,  we  obtain the estimates
$\eqref{integral3}$ and \eqref{integral4}.
\end{proof}

\begin{lemm}\label{derivative estimate1}
Let $ t \le (2r)^2$ and $|x'| \leq 2r$. Let $k_0\geq 1$.
\begin{enumerate}
\item
For  $x_n^2\leq t$ we have
\begin{align}
\label{a1}
|D^{k_0}_{x'} D_{x_n}^{l_0} u(x,t)| & \leq
\left\{\begin{array}{ll}
cr^{\al-\frac{n+1}p -1}  \ln ( 1 + {\frac{t}{x_n^2}} )&\mbox{if }l_0+k_0=1,\\
cr^{\al-\frac{n+1}p -1}x_n^{1-k_0-l_0}&\mbox{if }l_0+k_0\geq 2.
\end{array}\right.
\end{align}
\item
For   $  t \leq x_n^2$,  we have
\begin{align}
\label{a2}
|D^{k_0}_{x'} D_{x_n}^{l_0} u(x,t)|
& \leq
 \left\{\begin{array}{ll}  cr^{\al -\frac{n+1}p-1}   x_n^{-( k_0 +l_0 )   }t^{\frac{1}{2}}& x_n \leq r,\\
cr^{\al -\frac{n+1}p+n-2}   x_n^{-( n-1 + k_0 +l_0 )   }t^{\frac{1}{2}}& x_n \geq r.
 \end{array}
 \right.
\end{align}
\end{enumerate}
\end{lemm}
\begin{proof}
Since  $k_0 \geq 1$, we have
\begin{align} \label{derivative estimate 0-0-1}
|\notag   D_{x_n}^{l_0} D^{k_0}_{x'} u^i(x,t)|
&  = |\int_0^t \int_{\Rn}
                      D^{k_0 -1}_{x'}D^{l_0}_{x_n} K_{ij}(x'-y', x_n, t-s) D_{y'}a(y',s)dy'ds|\\
\notag&\leq  cr^{\al  -\frac{n+1}p-1} \int_0^t  \int_{|x'-y'|\leq 3r}
(t-s)^{-\frac12} (  x_n^2 + t-s )^{-\frac{l_0}2} \\
& \quad \times ( |x'-y'|^2 +x_n^2  +t-s )^{-\frac{n + k_0-1}2 } dy'ds\\
\notag&=  cr^{\al  -\frac{n+1}p-1} \int_0^t  \int_{|y'|\leq 3r}
s^{-\frac12} ( x_n^2 + s)^{-\frac{l_0}2} ( |y'|^2 + x_n^2 + s )^{-\frac{n + k_0-1}2 } dy'ds.
\end{align}
Note that
\begin{align}\label{0907}
\notag \int_{|x'-y'| \leq 3r}
                   \frac{1}{ (|x' -y'|^2+ x_n^2 + s)^{\frac{n + k_0 -1}2} } dy'
               & = c \int_0^{ 3r}  \frac{\rho^{n-2} }{ (\rho^2+ x_n^2 + s)^{\frac{n+k_0 -1}2}} d\rho\\
&   = c ( x_n^2 + s)^{-\frac{k_0}2}  \int_0^{ \frac{3r}{ \sqrt{x_n^2 + s}   }}  \frac{\rho^{n-2} }{ (\rho^2+ 1)^{\frac{n +k_0 -1}2} } d\rho\\
                   & \notag\leq  \left\{\begin{array}{ll}
                 c ( x_n^2 + s)^{-\frac{k_0}2}   &  r \geq x_n, \\
                 c ( x_n^2 + s)^{-\frac{k_0}2}   ( \frac{r}{ x_n  } )^{n-1} &  r \leq  x_n.
        \end{array}
        \right.
\end{align}
Taking $\te_2: = \frac{k_0 + l_0}2 \geq \frac12$ in \eqref{0822} - \eqref{0822-1} ,
we  obtain the estimates $\eqref{a1}$ and $\eqref{a2}$.
\end{proof}
\begin{lemm}\label{lemma2}
Let $t \leq (2r)^2$ and $ |x'| \leq 2r$.
 Let  $ l_0\geq 1$.
\begin{enumerate}
\item
For  $x_n^2\leq t$ we have
\begin{align}
\label{a5}
|D_{x_n}^{l_0} u(x,t)|
 &\leq\left\{\begin{array}{ll} \vspace{2mm}
cr^{\al -\frac{n+1}p-1}  \ln (1 + {\frac{r}{x_n}})   \ln ( 1 + {\frac{t}{x_n}}) &\mbox{if }l_0=1,\\
cr^{\al -\frac{n+1}p-1} x_n^{1-l_0}   \ln (1 + { \frac{r}{x_n} }) &\mbox{if }l_0\geq 2.
\end{array}\right.
\end{align}
\item
For  $ t \leq  x_n^2 $, then
\begin{align}
\label{a55}
|  D^{l_0}_{x_n} u (x,t) | \leq c r^{\al -\frac{n+1}p-1 } x_n^{-l_0}t^{\frac{1}{2}}
 \Big(  \chi_{ x_n \leq r}  (x_n)  \ln ( 1 + { \frac{r}{x_n}} )   + \chi_{ x_n \geq r}  (x_n)     rx_n^{-1}      \Big).
\end{align}
\end{enumerate}
\end{lemm}
\begin{proof}
 Decompose
$u^i$ by $u^i = u_1^i + u^i_2$, where
\begin{align}\label{u-1}
u^i_1(x,t) &= -2\de_{ij} \int_0^t \int_{\Rn} D_{x_n} \Ga(x'-y', x_n ,t-s)  a(y',s) dy'ds,\\
\label{u-2}
u^i_2(x,t) &=4 \int_0^t \int_{\Rn}   G_{ij} (x'-y', x_n, t-s) a (y', s) dy' ds.
\end{align}

First, we estimate $D_{x_n}^{l_0} u^i_1$.  Note that  for $x_n>0$, $D_{x_n}^2\Gamma=D_t\Gamma  -\De_{n-1}\Gamma$ in ${\mathbb R}^n_+ \times
(0\infty)$,
 where $\De_{n-1}=\sum_{k=1}^{n-1}D_{x_k}^2$.
Hence, if $l_0\geq 1$, then
\begin{align*}
D_{x_n}^{l_0} u^i_1(x,t) & =  -2\de_{ij} \int_0^t \int_{\Rn} D^{l_0 +1}_{x_n} \Ga(x'-y', x_n ,t-s)  a(y',s) dy'ds\\
& = -2\de_{ij} \int_0^t \int_{\Rn}  ( D_t  -\De_{n-1}) D^{l_0 -1}_{x_n} \Ga(x'-y', x_n ,t-s)  a(y',s) dy'ds\\
& = -2\de_{ij}  \int_0^t \int_{\Rn}  D^{l_0 -1}_{x_n} \Ga(x'-y', x_n ,t-s)  D_s a(y',s) dy'ds\\
& \quad-2\de_{ij}   \int_0^t \int_{\Rn}  D_{x_k} D^{l_0 -1}_{x_n} \Ga(x'-y', x_n ,t-s)   D_{y_k} a(y',s) dy'ds.
\end{align*}
Note that for each multi-index $m $, there is $c(m)> 0$ such that
 $| D_{x}^{ m} \Ga(x ,t)| \leq c(m)  t^{-\frac{n+ |m|}{2}}e^{-\frac{|x|^2 }{2t}}$.
Hence, for $t\leq (2r)^2$ and  $|x'| \leq 2r$, we have
\begin{align}
\label{a3}
\notag |D_{x_n}^{l_0} u^i_1(x,t)|& \leq
 c r^{\al -\frac{n+1}p -2} \int_0^t  \int_{|y'|< r}  (t-s)^{-\frac{n+l_0-1}{2}}e^{-\frac{|x'-y'|^2+x_n^2}{ 2(t-s)}}dy'ds \\
 \notag
   &\quad+cr^{\al -\frac{n+1}p-1} \int_0^t   \int_{|y'|<r}  (t-s)^{-\frac{n+l_0}{2}}e^{-\frac{|x'-y'|^2+x_n^2}{ 2(t-s)}}dy' ds\\
   & \leq
c r^{\al -\frac{n+1}p -2} \int_0^t  \int_{|y'|< 3r} s^{-\frac{n+l_0-1}{2}}e^{-\frac{|y'|^2+x_n^2}{ 2s}}dy'ds \\
\notag
  &\quad + cr^{\al -\frac{n+1}p-1} \int_0^t
 \int_{|y'|< 3r}  s^{-\frac{n+l_0}{2}}e^{-\frac{|y'|^2+x_n^2}{2s}}dy' ds\\
\notag
  & \leq  cr^{\al -\frac{n+1}p-2}   \int_0^t  s^{  -\frac{l_0}{2}}   e^{-\frac{x_n^2  }{ 2s} }ds
                                 +cr^{\al -\frac{n+1}p-1}  \int_0^t  s^{  -\frac{l_0+1}{2}  }  e^{-\frac{x_n^2  }{2s} }ds.
\end{align}
Note that $e^{-\frac{x_n^2  }{2s} } \leq c(\frac{x_n^2}{s  }+1
)^{-\frac{m}{2}}$ for any $m\geq 0$. Hence, for $x_n^2\leq t$ we
have
\begin{align*}
\int_0^t  s^{  -\frac{l_0}{2}  }  e^{-\frac{x_n^2  }{2s} }ds&\leq
c\int^t_0s^{-\frac{1}{2}}(x_n^2+s)^{-\frac{l_0-1}{2}}ds\leq cx_n^{-l_0+1}t^{\frac{1}{2}}\leq cx_n^{-l_0+1}r,\\
\int_0^t  s^{  -\frac{l_0+1}{2}  }  e^{-\frac{x_n^2  }{2s} }ds&\leq
c\int^t_0s^{-\frac{1}{2}}(x_n^2+s)^{-\frac{l_0}{2}}ds
\leq cx_n^{-l_0}\int^{x_n^2}_0s^{-\frac{1}{2}}ds+c\int^t_{x_n^2}  s^{-\frac{l_0+1}{2}}ds\\
 &\leq \left\{\begin{array}{ll}
c\ln ( 1 + {\frac{t}{x_n^2}} ) &\mbox{for }l_0=1,\\
cx_n^{-l_0+1} &\mbox{for }l_0\neq 1.
\end{array}\right.
\end{align*}
For $x_n^2\geq t$ and $a : = \frac{l_0}2$ or $a : = \frac{l_0 +1}2$,  we have
\begin{align*}
\int_0^t  s^{  -a }  e^{-\frac{x_n^2  }{ 2s} }ds&=
 x_n^{-2a +2} \int^\infty_{\frac{x_n^2}t} s^{a-2} e^{-\frac12s} ds
\leq cx_n^{-2a +2}  e^{-\frac{x_n^2}{4t}}.
\end{align*}
Applying the above estimates to the right hand side of \eqref{a3}, we have
for $x_n^2\leq t$
\begin{align}
\label{a4}
|D_{x_n}^{l_0} u^i_1(x,t)|
\leq \left\{\begin{array}{ll}
  cr^{\al -\frac{n+1}p-1} \ln ( 1 +{ \frac{t}{x_n^2}} ) &\mbox{for }l_0=1,\\
cr^{\al -\frac{n+1}p-1} x_n^{-l_0+1} &\mbox{for }l_0\geq 2
\end{array}\right.
\end{align}
and for $x_n^2\geq t$
\begin{align}
\label{a555}
|D_{x_n}^{l_0} u^i_1(x,t)|& \leq
cr^{\al -\frac{n+1}{p}-2} x_n^{-l_0 +2}  e^{-\frac{x_n^2}{4t}}  + r^{\al -\frac{n+1}{p}-1} x_n^{-l_0 +1}  e^{-\frac{x_n^2}{4t}}.
\end{align}

Next, we would like to estimate $D_{x_n}^{l_0}u^i_2$.
Observe that $\int_{\Rn} G_{ij} (y',x_n, t) dy' =0$ for $1 \leq i \leq n, \,\, 1 \leq j \leq n-1$.
Hence,
\begin{align*}
D^{l_0}_{x_n} u_2^i (x,t) &= \int_0^t \int_{\Rn}   D^{l_0}_{x_n}   G(x' -y',x_n, t-s) (a(y',s) - a(x',s) ) dy'ds.
\end{align*}
From the estimate  \eqref{inequality11} of $G_{ij}$, we have that for $t \leq (2r)^2$ and $ |x'| \leq 2r$
\begin{align} \label{a6}
\notag|D^{l_0}_{x_n} u_2^i (x,t)|&\leq c r^{\al -\frac{n+1}p -1} \int_0^t \int_{|y'| \leq 3r}
                   \frac{|x' -y'|}{(t-s)^\frac12 (|x' -y'|^2+ x_n^2 + t-s)^{\frac{n}2} ( x_n^2 + t-s)^{\frac{l_0}2 }} dy'ds\\
& \notag  \quad  +   r^{\al -\frac{n+1}p } \int_0^t \int_{|y'| \geq 3r}
                   \frac{1}{(t-s)^\frac12 (|x' -y'|^2+ x_n^2 + t-s)^{\frac{n}2} ( x_n^2 + t-s)^{\frac{l_0}2 }} dy'ds\\
&\leq c r^{\al -\frac{n+1}p -1} \int_0^t \int_{|x'-y'| \leq 5r}
                   \frac{|x' -y'|}{s^\frac12 (|x' -y'|^2+ x_n^2 + s)^{\frac{n}2} ( x_n^2 + s)^{\frac{l_0}2 }} dy'ds\\
&  \quad +  c r^{\al -\frac{n+1}p } \int_0^t \int_{|x'-y'| \geq r}
                   \frac{1}{s^\frac12 (|x' -y'|^2+ x_n^2 + s)^{\frac{n}2} ( x_n^2 + s)^{\frac{l_0}2 }} dy'ds.\notag
\end{align}
Note that
\begin{align}\label{0822-2}
\notag \int_{|x'-y'| \leq 5r}
                   \frac{|x' -y'|}{ (|x' -y'|^2+ x_n^2 + s)^{\frac{n}2} } dy'
               &\leq c \int_0^{ 5r}  \frac{\rho^{n-1} }{ (\rho^2+ x_n^2 + s)^n} d\rho\\
&   = c \int_0^{ \frac{5r}{ \sqrt{x_n^2 + s}   }}  \frac{\rho^{n-1} }{ (\rho^2+ 1)^{\frac{n}2}} d\rho\\
                   & \notag \leq c \ln ( 1 + {\frac{r}{ x_n}}).
\end{align}
And
\begin{align}\label{0822-3}
\int_{|x'-y'| \geq r}  \frac{1}{ (|x' -y'|^2+ x_n^2 + s)^{\frac{n}2} } dy'
    & \leq  c \int_r^\infty \frac{ 1 }{ ( \rho+ \sqrt{ x_n^2 + s})^2 } d\rho \leq   c (  x_n+  r  )^{-1}.
\end{align}
Applying the \eqref{0822-2} and \eqref{0822-3} to the right hand side of \eqref{a6}, we have
\begin{align}
\label{a7}
|D^{l_0}_{x_n} u_2^i (x,t)|\leq  c r^{\al -\frac{n+1}p -1}
\Big(  \chi_{ x_n \leq r}  (x_n)  \ln ( 1 + { \frac{r}{x_n}} )   + \chi_{ x_n \geq r}  (x_n)     rx_n^{-1}      \Big)
\int_0^t s^{-\frac12 } ( x_n^2 + s)^{-\frac{l_0}2 }
ds.
\end{align}
Taking $\te_2 = \frac{l_0}2 \geq \frac12$ in  \eqref{0822} and \eqref{0822-1}, from \eqref{a7}, we have  that for $x_n^2\leq t\leq (2r)^2$
\begin{align}
\label{a8}
&|D^{l_0}_{x_n} u_2^i (x,t)|\leq   \left\{\begin{array}{ll}
c r^{\al -\frac{n+1}p -1}     \ln( 1 + { \frac{r}{x_n}} )   \ln(1 + { \frac{t}{x^2_n}})
   &\mbox{ for }l_0=1,\\
c r^{\al -\frac{n+1}p -1}     \ln ( 1 + { \frac{r}{x_n}} )
x_n^{1-l_0} &\mbox{ for }l_0\geq 2,
\end{array}\right.
\end{align}
and for $x_n^2\geq t$
\begin{align}
\label{a9}
|D^{l_0}_{x_n} u_2^i (x,t)|\leq c r^{\al -\frac{n+1}p-1}x_n^{-l_0}t^{\frac{1}{2}}
   \Big(  \chi_{ x_n \leq r}  (x_n) ( 1+\ln{ \frac{r}{x_n}} )   + \chi_{ x_n \geq r}  (x_n)     rx_n^{-1}      \Big).
\end{align}
From \eqref{a4}, \eqref{a555},  \eqref{a8} and \eqref{a9}, we obtain \eqref{a5} and \eqref{a55}.
\end{proof}

\section{ Estimates of  solution in  $ {\mathbb R}^n_+ \times (0,\infty) $}\label{anisotropicsection}

\setcounter{equation}{0}

\begin{theo}\label{Rn}
Let $1< p<\infty$ and  $0<\alpha<1 $.
Let  $g=(g',0) \in {\mathcal B}^{\al,\frac12\al }_p ({\mathbb R}^{n-1}
\times (0,\infty))$ with $g(x', 0) =0$ if $\al > \frac2p$.
Then there is a solution $u$ of the Stokes system \eqref{maineq2} in ${\mathbb R}^{n-1}
\times (0,\infty)$ with boundary data $g$ such that
\begin{align}
\label{m3}
\int_0^\infty \int_{{\mathbb R}^n_+ }(x_n \wedge t^\frac12)^{-\al p+( k_0 +l_0 + 2m_0 )p-1}
                        |D_{x_n}^{l_0} D_{x'}^{k_0} D_t^{m_0}  u(x,t)|^pdxdt
\leq c \| g\|^p_{{\mathcal B}^{\al, \frac\al2 }_p(\Rn \times (0,\infty) )}
\end{align}
 for $ k_0 + l_0 + 2m_0  =1,2,3.$

\end{theo}

\begin{proof}

By Proposition \ref{atom decomposition}, it is sufficient to
consider $g$ which consists of  the  $(\al,p)$ atoms. For
simplicity, assume  $g_k=a\delta_{kj}, j=1,\cdots, n-1$ for some
atom $a$ supported on $\Delta(0, r^2) \times (0, r^2)$. We denote
$\be = -\al p+( k_0 +l_0 + 2m_0 )p-1   $. Let  $u$ be defined by
\eqref{sample1}. Decompose the domain of integration into four parts
so that
\begin{align*}
  \int_0^\infty  \int_0^\infty\int_{\Rn } (x_n \wedge t^\frac12)^\be
             | D^{l_0}_{x_n} D^{k_0}_{x'} D^{m_0}_t  u|^p dx' dx_n dt=\sum_{i=1}^4I_i,
             \end{align*}
             where
             \begin{align*}
               I_1&= \int_0^{(2 r)^2} \int_0^\infty  \int_{|x'| \leq 2r} (x_n \wedge t^\frac12)^\be
                           | D^{l_0}_{x_n} D^{k_0}_{x'} D^{m_0}_t  u|^p dx' dx_n dt,\\
                           I_2&=\int_{(2 r)^2}^\infty  \int_0^\infty  \int_{|x'| \leq 2r} (x_n \wedge t^\frac12)^\be
             | D^{l_0}_{x_n} D^{k_0}_{x'} D^{m_0}_t  u|^p dx' dx_n dt, \\
  I_3&=  \int_0^{(2  r)^2} \int_0^\infty  \int_{|x'| \geq 2r} (x_n \wedge t^\frac12)^\be
             | D^{l_0}_{x_n} D^{k_0}_{x'} D^{m_0}_t  u|^p dx' dx_n dt,\\
 I_4&=  \int_{(2  r)^2}^\infty  \int_0^\infty  \int_{|x'| \geq 2r} (x_n \wedge t^\frac12)^\be
             | D^{l_0}_{x_n} D^{k_0}_{x'} D^{m_0}_t  u|^p dx' dx_n dt.
\end{align*}

$\bullet$
 From the estimate $\eqref{EQ1}_1$,    we have
\begin{align}\label{0721-1}
I_{4} & \leq  cr^{p \al -n-1+(n+1)p } \int_{(2  r)^2}^\infty \int^\infty_0  \int_{|x'| \geq 2r    }(x_n\wedge t^\frac12)^\be
\notag \\
&\qquad\qquad\qquad \times
            t^{-(\frac12 +m_0) p}(x_n^2 +t )^{-\frac{l_0p}{2}} ( |x'|^2 + x_n^2+t)^{-\frac{(n+k_0)p}{2}}dx'dx_ndt \notag \\
& \leq   cr^{p \al -n-1+ (n+1)p } \int_{(2  r)^2}^\infty \int^\infty_0(x_n \wedge t^\frac12)^\be
 t^{-(\frac12 +m_0) p}(x_n^2 +t)^{-\frac{(n+k_0+l_0)p}{2}+\frac{n-1}{2}}dx_ndt.
\end{align}
Note that for $\te_1  > -1 $ and $ \te_2 > \frac12 $, we have
\begin{align} \label{0721-2}
 \notag \int^\infty_0 (x_n \wedge t^\frac12)^{\te_1 }(x_n^2 +t )^{-\te_2}dx_n
&\leq t^{-\te_2}\int_0^{t^{\frac{1}{2}}} x_n^{\te_1}dx_n
+t^{\frac{\te_1}{2} }\int^\infty_{t^{\frac{1}{2}}} x_n^{-2 \te_2}dx_n\\
&=ct^{ -\te_2 +\frac12 \te_1 + \frac12}.
\end{align}
Hence, taking $\te_1 : = \be >-1$ and $\te_2 :  = \frac{(n+k_0+l_0)p}{2}-\frac{n-1}{2} > \frac12  $ in \eqref{0721-2}, from \eqref{0721-1}, we have
\begin{align}\label{1004-1}
I_{4}
& \leq   cr^{p \al -n-1+(n+1) p } \int_{(2  r)^2}^\infty t^{\frac{n-1}{2}-\frac{(n+\al+1)p}{2}}dt=c\quad \mbox{ for } \al<k_0+l_0+2m_0\mbox{ and }p>1.
\end{align}

$\bullet$
 From the estimate $\eqref{EQ1}_2$,    we have
\begin{align*}
I_{2} & \leq  cr^{p \al -n-1+(n+1) p } \int_{(2 r)^2}^\infty  \int_0^\infty  \int_{|x'| \leq 2r}(x_n\wedge t^\frac12)^\be
                       t^{-(\frac12 +m_0) p} (x_n^2 +t)^{-\frac{(n+k_0+l_0)p}{2}}dx'dx_ndt\\
&=cr^{p \al -2+(n+1) p} \int_{(2 r)^2}^\infty  \int_0^\infty (x_n\wedge t^\frac12)^\be
                                       t^{-(\frac12 +m_0) p} (x_n^2 +t)^{-\frac{(n+k_0+l_0)p}{2}} dx_ndt.
\end{align*}
Hence, taking $\te_1 : = \be > -1$ and $\te_2: = \frac{ n + k_0 + l_0 }{2} > \frac12$ in \eqref{0721-2}, we have
\begin{align}\label{1004-2}
I_{2}
& \leq   cr^{p \al -2+p(n+1) } \int_{(2  r)^2}^\infty t^{-\frac{(n+\al+1)p}{2}}dt=c.
\end{align}

$\bullet$
For the estimate of $I_1$, we divide the domain of integration so that
\begin{align}\label{1004-3}
I_1=I_{11}+I_{12},
\end{align}
where
\begin{align*}
I_{11}= \int_0^{(2 r)^2} \int_0^{t^{\frac{1}{2}}} \int_{|x'| \leq 2r} x_n^\be
                           | D^{l_0}_{x_n} D^{k_0}_{x'} D^{m_0}_t  u|^p dx' dx_n dt,\\
                           I_{12}= \int_0^{(2 r)^2} \int_{t^{\frac{1}{2}}}^\infty \int_{|x'| \leq 2r}   t^{\frac12 \be}
                           | D^{l_0}_{x_n} D^{k_0}_{x'} D^{m_0}_t  u|^p dx' dx_n dt.
\end{align*}

First we estimate $I_{12}$.\\
 \emph{1). Let $m_0\geq 1$.}
  From the estimate  \eqref{integral2} we have
\begin{align*}
I_{12}
&\leq  cr^{\al p-n-1-2m_0 p}\int_0^{(2 r)^2} \int^\infty_{t^{\frac{1}{2}}}  \int_{|x'| \leq 2r}  t^{\frac12 \be}
                                          x_n^{-( k_0+l_0+1)p }t^{\frac{p}{2}} dx' dx_n dt\\
                          &= cr^{\al p-2-2m_0 p} \int_0^{(2r)^2} \int^\infty_{t^{\frac{1}{2}}}
                          t^{\frac12 \be + \frac{p}2}
                          x_n^{- (k_0+l_0+1)p }  dx_n dt\\
&\leq  cr^{\al p-2-2m_0p}\int_0^{(2 r)^2}t^{-\frac{\al
p}{2}+m_0p}dt=c \qquad \mbox{ for }\al<2m_0+\frac{2}{p}.
\end{align*}

\emph{2). Let  $m_0=0$ and $k_0\geq 1$.}
From the estimate  \eqref{a2} we have
\begin{align*}
I_{12}&\leq  cr^{\al p-n-1- p}\int_0^{(2 r)^2} \int^r_{t^{\frac{1}{2}}}  \int_{|x'| \leq 2r} t^{\frac12 \be}
                          x_n^{- ( k_0+l_0)p }t^{\frac{p}{2}} dx' dx_n dt\\
&    \quad +  cr^{\al p-n-1 + (n-2)p}\int_0^{(2 r)^2} \int^\infty_r  \int_{|x'| \leq 2r} t^{\frac12 \be}
                          x_n^{- (  n-1 + k_0+l_0)p }t^{\frac{p}{2}} dx' dx_n dt\\
                          &\leq cr^{\al p-2- p} \int_0^{(2r)^2} \int^r_{t^{\frac{1}{2}}}
                          t^{\frac12 \be + \frac{p}2}   x_n^{- (k_0+l_0)p }  dx_n dt\\
&    \quad +  cr^{\al p-2 + (n-2)p}\int_0^{(2 r)^2} \int^\infty_r    t^{\frac12 \be + \frac12p}
                          x_n^{- (  n-1 + k_0+l_0)p }   dx_n dt\\
&\leq  cr^{\al p-2- p}\int_0^{(2 r)^2}t^{-\frac{\al p}{2}+\frac{p}{2}}dt
            +  cr^{\al p-( 1 + k_0 + l_0) p -1}\int_0^{(2 r)^2}    t^{\frac12 \be + \frac12p}
                          dt\\
&=c \qquad \mbox{ for }\al<1+\frac{2}{p}.
\end{align*}

\emph{3). Let $m_0=k_0=0$ and $l_0\geq 1$.}
From the estimate  \eqref{a55} we have
\begin{align}\label{0907-1}
\begin{array}{ll}\vspace{2mm}
I_{12}&\leq  cr^{\al p-n-1-p}\int_0^{(2 r)^2} \int^r_{t^{\frac{1}{2}}}  \int_{|x'| \leq 2r} t^{\frac12 \be}
                        x_n^{-l_0 p}    \ln^p ( 1 + { \frac{r}{x_n}})  t^{\frac{p}{2}}dx' dx_n dt\\ \vspace{2mm}
    & \quad + cr^{\al p-n-1}\int_0^{(2 r)^2} \int^\infty_r  \int_{|x'| \leq 2r} t^{\frac12 \be}
                        x_n^{-l_0 p}  x_n^{-p} t^{\frac{p}{2}}dx' dx_n dt\\     \vspace{2mm}
       &\leq  cr^{\al p-p-2} \int_0^{(2 r)^2}  \int^r_{t^{\frac{1}{2}}}  t^{\frac12 \be +\frac{p}2 }
                                  x_n^{-l_0 p}  \ln^p ( 1 + { \frac{r}{x_n}}) dx_n dt\\
    & \quad + cr^{\al p-2}\int_0^{(2 r)^2} \int^\infty_r    t^{\frac12 \be + \frac{p}{2}}
                        x_n^{-( l_0 +1) p}  dx_n dt.
\end{array}
\end{align}
Here
\begin{align}\label{0907-2}
\int_0^{(2 r)^2} \int^\infty_r    t^{\frac12 \be + \frac{p}{2}} x_n^{-( l_0 +1) p}  dx_n dt
   = c  r^{-(l_0 +1)p +1} \int_0^{(2 r)^2}     t^{\frac12 \be + \frac{p}{2}}   dt = c r^{-\al p +2}.
\end{align}
Note that for $0< \ep  < 1$ there is a positive constant $c(\ep) >0$ such that for $a > 0$
 \begin{equation}
 \label{077}\ln ( 1 + a)  \leq c(\ep)a^\ep.
 \end{equation}
Hence, taking $\ep >0$ sufficiently small  in \eqref{077} ($ \ep < 1 + \frac2p -\al$),    we have
 \begin{align} \label{0907-3}
 \begin{array}{ll} \vspace{2mm}
 \int_0^{(2 r)^2}  \int^r_{t^{\frac{1}{2}}}  t^{\frac12 \be +\frac{p}2 }
                          x_n^{-l_0 p}  \ln^p ( 1 + { \frac{r}{x_n}}) dx_n dt &\leq  c(\ep)    \int_0^{(2 r)^2}  \int_{t^\frac12}^{ r}
       t^{\frac12 \be +\frac{p}2}
                          x_n^{-l_0 p   }  ( \frac{r}{x_n})^{\ep p}   dx_n dt        \\  \vspace{2mm}
& \leq  c (\ep) \int_0^{(2 r)^2}t^{ \frac12\be + \frac{p}2}   r^{\ep p}   t^{-\frac{l_0}2p + \frac12 -\frac12 \ep p}   dt \\
& =c(\ep) r^{ -\al p+p +2 }.
\end{array}
\end{align}
From \eqref{0907-1}, \eqref{0907-2} and \eqref{0907-3},  we obtain
\begin{align}\label{0907-4}
I_{12} \leq c.
\end{align}

Second, we estimate $I_{11}$.\\
 \emph{1). Let  $m_0\geq 1$ and $k_0+l_0\geq 1$.} Then from the estimate $\eqref{integral1}_1$ we have
\begin{align*}
I_{11}&\leq  cr^{\al p-n-1-2m_0 p}\int_0^{(2 r)^2} \int_0^{t^{\frac{1}{2}}} \int_{|x'| \leq 2r} x_n^{ \be }
                          x_n^{-(k_0+l_0)p} dx' dx_n dt\\
  & \leq  cr^{\al p-2-2m_0 p}\int_0^{(2 r)^2}t^{-\frac{\al p}{2}+m_0p}dt=c\mbox{ for }\al<2m_0.
\end{align*}

 \emph{2).
 Let  $m_0\geq 1$ and $k_0=l_0=0$}.  From the estimate $\eqref{integral1}_2$ and
 taking $\ep >0$ sufficiently small  in \eqref{077} ($ \frac{\ep}p < 2m_0 + \frac1p -\al$), we have
\begin{align*}
I_{11}&\leq  c r^{\al p-n-1-2m_0 p}\int_0^{(2 r)^2} \int_0^{t^{\frac{1}{2}}}  \int_{|x'| \leq 2r}
x_n^{\beta}  \ln^p (1 +  {\frac{t}{x_n^2}} )  dx'dx_ndt\\
&\leq  c(\ep) r^{\al p-2-2m_0 p}\int_0^{(2 r)^2} \int_0^{t^{\frac{1}{2}}}
                        x_n^{\be}  (\frac{t}{x_n^2})^{\ep p}    dx_n dt\\
& \leq  c(\ep) r^{\al p-2-2m_0 p}\int_0^{(2 r)^2}t^{-\frac{\al p}{2}+m_0p}dt=c(\ep).  
\end{align*}

\emph{3). Let $m_0=0$ and $k_0\geq  1, k_0+l_0\geq 2$.} Then from the estimate $\eqref{a1}_2$,  we have
\begin{align*}
I_{11}&\leq  cr^{\al p-n-1-p}\int_0^{(2 r)^2} \int_0^{t^{\frac{1}{2}}} \int_{|x'| \leq 2r}
                      x_n^{ \be } x_n^{( 1-k_0-l_0)p} dx' dx_n dt\\
& \leq  cr^{\al p-2- p}\int_0^{(2 r)^2}t^{-\frac{\al p}{2}+\frac{p}{2}}dt=c\mbox{ for }\al<1.
\end{align*}

\emph{4)}.
 Let $m_0=0$ and $k_0 = 1$.  Then from the estimate  $\eqref{a1}_1$ and \eqref{077} ( taking $\ep$ satisfying
  $0<\frac{\ep}p < 1 -\al$),  we have
\begin{align*}
I_{11}&\leq cr^{\al p -n-1-p}\int^{(2r)^2}_0\int^{t^{\frac{1}{2}}}_0\int_{|x'|\leq 2r}x_n^{\beta}  \ln^p ( 1 +  {\frac{t}{x_n^2}} ) dx'dx_ndt\\
&\leq  c(\ep)r^{\al p-2-p}\int_0^{(2 r)^2} \int_0^{t^{\frac{1}{2}}}
                       x_n^{ \be }  (\frac{t}{x_n^2})^{\ep p}      dx_n dt\\
 & \leq  c(\ep)r^{\al p-2- p}\int_0^{(2 r)^2}t^{-\frac{\al p}{2}+\frac{p}{2}}dt=c(\ep).
\end{align*}

 \emph{5). Let $m_0=k_0=0$ and $l_0\geq 2$.} Then from the estimate $\eqref{a5}_2$ and \eqref{077} (taking $\ep$ satisfying
 $0< \frac{\ep}p < 1 -\al$ ),  we have
 \begin{align*}
I_{11}&\leq  cr^{\al p-n-1- p }\int_0^{(2 r)^2} \int_0^{t^{\frac{1}{2}}} \int_{|x'| \leq 2r}x_n^{\beta}x_n^{p-l_0p}
                         \ln^p ( 1 + { \frac{r}{x_n}} )dx'dx_ndt\\
&\leq c(\ep) r^{\al p-2- p }\int_0^{(2 r)^2} \int_0^{t^{\frac{1}{2}}} x_n^{\beta+ p-l_0p}  (\frac{r}{x_n})^{\ep p} dx_ndt\\
 & \leq  c(\ep)r^{\al p-2- p } \int_0^{(2 r)^2}t^{-\frac{\al p}{2}+\frac{p}{2} }  r^{\ep p} t^{-\frac\ep2 p}     dt=c(\ep).
\end{align*}

 \emph{6). Let  $m_0=k_0=0$ and $l_0= 1$.} Then from the estimate $\eqref{a5}_1$  and  and \eqref{077} (taking $\ep$ satisfying
 $0< \frac{\ep}p < 1 -\al$ ),  we have
 \begin{align*}
I_{11}&\leq  cr^{\al p-n-1- p  }\int_0^{(2 r)^2} \int_0^{t^{\frac{1}{2}}} \int_{|x'| \leq 2r}
x_n^\beta  \ln^p ( 1 +  {\frac{r}{x_n} } ) dx' dx_n dt\\
 &\leq c(\ep)r^{\al p-2- p  }\int_0^{(2 r)^2} \int_0^{t^{\frac{1}{2}}}  x_n^\beta   (\frac{r}{x_n} )^{\ep p}  dx_n dt\\
& \leq  c(\ep)r^{\al p-2- p  }\int_0^{(2 r)^2}t^{\frac12 \be +\frac12}   r^{\ep p} t^{ - \frac12 \ep p} dt=c(\ep).
\end{align*}
From  1) to 6), we get that $I_{11} \leq c$ and hence with  \eqref{0907-4} and \eqref{1004-3}, we get
\begin{align}\label{1004-5}
I_1 \leq c.
\end{align}

$\bullet$
For the estimate of $I_3$, we divide the domain of integration so that  $$
I_3=I_{31}+I_{32},$$
where
\begin{align*}
I_{31}= \int_0^{(2 r)^2} \int_0^{t^{\frac{1}{2}}} \int_{|x'| \geq 2r} x_n^{ \be }
                           | D^{l_0}_{x_n} D^{k_0}_{x'} D^{m_0}_t  u|^p dx' dx_n dt,\\
                           I_{32}= \int_0^{(2 r)^2} \int_{t^{\frac{1}{2}}}^\infty \int_{|x'| \geq 2r}
                                    t^{ \frac12 \be}
                           | D^{l_0}_{x_n} D^{k_0}_{x'} D^{m_0}_t  u|^p dx' dx_n dt.
\end{align*}
First, we estimate $I_{32}$. From the estimate  \eqref{integral4} we have
\begin{align*}
I_{32}
&\leq  cr^{\al p-n-1-2m_0 p+(n-1)p}
  \int_0^{(2 r)^2} \int^\infty_{t^{\frac{1}{2}}}  \int_{|x'| \geq 2r}
t^{\frac12 \be }  x_n^{ -  l_0p }(|x'|^2+x_n^2)^{-\frac{(n+k_0)p}{2}}t^{\frac{p}{2}} dx' dx_n dt.
                         \end{align*}
 Note that for $\ga < - \frac{n-1}2 $ and $A \geq  0$, we have
\begin{align}\label{0721-4}
\int_{|x'| \geq 2r}(|x'|^2+A^2 )^{ \ga  }dx'
\leq c \int^\infty_{ 2r}(\rho+A )^{ 2\ga +n -2  }d\rho \leq c (A + r)^{2\ga +n-1}.
\end{align}
Taking $\ga : =-\frac{(n+k_0)p}{2} $ and $A: = x_n $ in \eqref{0721-4}, we have
\begin{align}\label{1004-5}
\begin{array}{ll}\vspace{2mm}
 I_{32} &\leq  cr^{\al p- (1 + k_0 +2m_0)p -2} \int_0^{(2 r)^2} \int^r_{t^{\frac{1}{2}}}
                            t^{ \frac12 \be + \frac12p  }  x_n^{-l_0p }dx_ndt\\ \vspace{2mm}
&\quad +   cr^{\al p-n-1-2m_0 p+(n-1)p}\int_0^{(2 r)^2} \int^\infty_r
                      t^{ \frac12 \be  + \frac12 p }  x_n^{ -(n+k_0 + l_0)p+n-1   } dx_ndt\\
& \leq c.
\end{array}
\end{align}
Now, we estimate $I_{31}$.\\
\emph{1). Let $l_0\neq 1$.}  From the estimate $\eqref{integral3}_2$, we have
\begin{align*}
I_{31}\leq  cr^{ \al p -n-1 -2m_0 p + (n-1)p}   \int_0^{(2 r)^2} \int_0^{t^{\frac{1}{2}}} \int_{|x'| \geq 2r}
                    x_n^{\beta} ( x_n^2 +t)^{ \frac{1 -l_0}2p } (|x'|^2+x_n^2)^{\frac{-n-k_0}2p } dx' dx_n dt.
\end{align*}
Taking $\ga: = -\frac{(n+ k_0)p}2 < -\frac{n-1}2  $ and $A: = x_n$
 in \eqref{0721-4},   we have
\begin{align*}
I_{31}
&\leq  cr^{\al p -n-1 -2m_0 p + (n-1)p}\int_0^{(2 r)^2} \int_0^{t^{\frac{1}{2}}} x_n^{\be}  t^{\frac{1-l_0}2 p} (x_n+r)^{n-1-(n+k_0)p}  dx_n dt\\
&\leq  cr^{(\al-2m_0-k_0 -1) p-2 }\int_0^{(2 r)^2}t^{ \frac{( -\al+
k_0+2m_0+1)p}{2}} dt=c.
\end{align*}

\emph{2). Let $l_0=1$.}
From the estimate  $\eqref{integral3}_1$, we have
\begin{align*}
I_{31}&\leq  cr^{(\al-2m_0) p-n-1+(n-1)p}\int_0^{(2 r)^2} \int_0^{t^{\frac{1}{2}}}  \int_{|x'| \geq 2r}x_n^{\beta}
                         (|x'|^2+x^2_n)^{-\frac{n+k_0}2 p}  \ln^p ( 1 + { \frac{t}{x_n^2}} ) dx' dx_n dt\\
  &\leq  c(\ep)r^{(\al-2m_0) p-n-1+(n-1)p}\int_0^{(2 r)^2} \int_0^{t^{\frac{1}{2}}}  \int_{|x'| \geq 2r}x_n^{\beta}
                          |x'|^ {-( n+k_0) p}    (\frac{t}{x_n^2})^{\ep p} dx' dx_n dt\\
&\leq  c(\ep) r^{(\al-2m_0-k_0 -1) p-2}\int_0^{(2 r)^2}
\int_0^{t^{\frac{1}{2}}}
              x_n^{ \be  }   (\frac{t}{x_n^2})^{\ep p}  dx_n dt\\
 & \leq  c(\ep)r^{(\al-2m_0-k_0 -1) p-2}\int_0^{(2 r)^2}t^{-\frac{\al p}{2}+\frac{(1+k_0+2m_0)p}{2}}dt=c(\ep).
\end{align*}
From 1) to 3) and \eqref{1004-5}, we get
\begin{align}\label{1004-6}
I_3 \leq c.
\end{align}
Hence, from \eqref{1004-1}, \eqref{1004-2}, \eqref{1004-5} and \eqref{1004-6}, we get \eqref{m3}.
\end{proof}

\begin{theo}\label{Rn4}

Let $T>0$. Let $ \beta>-1$.
Suppose that  $g\in {\mathcal B}^{\al,\frac{\al}{2}}_p({\mathbb R}^{n-1}\times (0,T))$.   Then
\begin{align*}
\int_0^T \int_{{\mathbb R}^{n}_+} (x_n\wedge t^{\frac{1}{2}})^{\beta }|  u|^pdx' dx_n dt
\leq cT^{\al p+\beta+1} \| g \|^p_{{\mathcal B}^{\al,\frac{\al}{2}}_p({\mathbb R}^{n-1}\times (0,T))}.
\end{align*}
\end{theo}

\begin{proof}
Without loss of generality, we assume   $g\in {\mathcal B}^{\al,\frac{\al}{2}}_p({\mathbb R}^{n-1}\times {\mathbb R})$
(Otherwise, there is, $\tilde{g}  \in {\mathcal B}^{\al,\frac{\al}{2}}_p({\mathbb R}^{n-1}\times {\mathbb R})$ so that
$\tilde{g}|_{{\mathbb R}^{n-1}\times (0,T)}=g$ and
$\|\tilde{ g} \|_{{\mathcal B}^{\al,\frac{\al}{2}}_p({\mathbb R}^{n-1}\times {\mathbb R})}
\leq c\| g \|_{{\mathcal B}^{\al,\frac{\al}{2}}_p({\mathbb R}^{n-1}\times (0,T))}$).
By Proposition \ref{atom decomposition}, it is sufficient to consider $g$ which consists of  the  $(\al,p)$ atoms (Note that we can
even take atoms of $g$  which  are supported in ${\mathbb R}^{n-1}\times (-T,2T)$).
 For simplicity, assume  $g_k=a\delta_{kj}, j=1,\cdots, n-1$ for some atom $a$ supported on $\Delta(0, r) \times (0 ,r^2)\subset B_T$.
 Let $u$ be represented by  \eqref{sample1}.
Decompose the domain of integration into four parts so that
\begin{align*}
 \int_0^T \int_0^\infty\int_{\Rn } (x_n \wedge t^\frac12)^{\beta}
             |  u|^p dx' dx_n dt             =\sum_{i=1}^4I_i,
             \end{align*}
             where
             \begin{align*}
               I_1&= \int_0^{(2 r)^2} \int_0^\infty  \int_{|x'| \leq 2r} (x_n \wedge t^\frac12)^{\beta}
                           |  u|^p dx' dx_n dt,\\
                           I_2&=\int_{(2 r)^2}^T  \int_0^\infty  \int_{|x'| \leq 2r} (x_n \wedge t^\frac12)^{\beta}
             |  u|^p dx' dx_n dt, \\
  I_3&=  \int_0^{(2  r)^2} \int_0^\infty \int_{|x'| \geq 2r} (x_n \wedge t^\frac12)^{\beta}
             |  u|^p dx' dx_n dt,\\
 I_4&=  \int_{(2  r)^2}^T \int_0^\infty  \int_{|x'| \geq 2r} (x_n \wedge t^\frac12)^{\beta}
             |  u|^p dx' dx_n dt.
\end{align*}

$\bullet$
 From the estimate $\eqref{EQ1}_1$ and \eqref{0721-4} ( taking $\ga: = -\frac{np}2  $ and $A: = \sqrt{x_n^2 +t}$ ), we have
\begin{align*}
I_{4} & \leq  cr^{p \al -n-1+(n+1)p } \int_{(2  r)^2}^T\int^\infty_0
\int_{|x'| \geq 2r    }(x_n\wedge t^\frac12)^{\beta} t^{-\frac12 p}(|x'|^2 +x_n^2+t)^{-\frac{np}{2}}dx'dx_ndt\\
&  \leq   cr^{p \al -n-1+p(n+1) } \int_{(2  r)^2}^T \int^\infty_0(x_n \wedge t^\frac12)^{\beta}
               t^{-\frac12 p}(x_n^2 +t+r^2)^{-\frac{np}{2}+\frac{n-1}{2}}dx_ndt.
\end{align*}
Taking $\te_1: = \be > -1$ and $\te_2: = \frac{np}2 - \frac{n-1}2 > \frac12$ in \eqref{0721-2}, we have
\begin{align*}
I_{4} &\leq  cr^{p \al -n-1+p(n+1) } \int_{(2  r)^2}^T t^{\frac{n+\beta-np}{2}}
               t^{-\frac12 p}dt\\
&=\left\{\begin{array}{ll}
cr^{p\al +1+\beta} & \mbox{ for }\beta<(n+1)p-n-2\\
cr^{p \al -n-1+p(n+1) }T^{\frac{n-(n+1)p+\beta}{2}+1}& \mbox{ for }\beta>(n+1)p-n-2\\
cr^{p \al -n-1+p(n+1) }\ln\frac{T}{(2r)^2} & \mbox{ for }\beta=(n+1)p-n-2
\end{array}\right\}\leq c T^{\frac{p\al +1+\beta}{2}}.
\end{align*}

$\bullet$ From the estimate $\eqref{EQ1}_2$,    we have
 \begin{align*}
I_{2} & \leq  cr^{p \al -n-1+p(n+1) } \int_{(2 r)^2}^T \int^\infty_0 \int_{|x'| \leq 2r}(x_n\wedge t^\frac12)^{\beta}t^{-\frac{p}2 } (x_n^2 +t)^{-\frac{np}{2}}dx'dx_ndt\\
&=  cr^{p \al -2+p(n+1) } \int_{(2 r)^2}^T  \int^\infty_0 (x_n\wedge t^\frac12)^{\beta} t^{-\frac{p}2 } (x_n^2 + t)^{-\frac{np}{2}}dx_ndt.
\end{align*}
Taking $\te_1: =\be> -1$
and $\te_2: = \frac{np}{2}>\frac12$ in \eqref{0721-2}, we have
 \begin{align*}
I_{2} & \leq   cr^{p \al -2+p(n+1) } \int_{(2  r)^2}^T t^{\frac{\beta-(n+1)p+1}{2}}dt\\
&=\left\{\begin{array}{l}
cr^{p\al +1+\beta}\mbox{ for }\beta<(n+1)p-3\\
cr^{p \al -2+p(n+1) }T^{\frac{3-(n+1)p+\beta}{2}+1}\mbox{ for }\beta>(n+1)p-3\\
cr^{p \al -2+p(n+1) }\ln\frac{T}{(2r)^2}\mbox{ for }\beta=(n+1)p-3
\end{array}\right\}\leq c T^{\frac{p\al +1+\beta}{2}}.
\end{align*}
$\bullet$
For the estimate of $I_1$, we divide the domain of integration so that  $$
I_1=I_{11}+I_{12},$$
where
\begin{align*}
I_{11}= \int_0^{(2 r)^2} \int_0^{t^{\frac{1}{2}}} \int_{|x'| \leq 2r} x_n^{\beta}
                           | u|^p dx' dx_n dt,\\
                           I_{12}= \int_0^{(2 r)^2} \int^\infty_{t^{\frac{1}{2}}} \int_{|x'| \leq 2r} t^{\frac{\beta}{2}}
                           |  u|^p dx' dx_n dt.
                                                      \end{align*}

From the estimate  \eqref{integral2} we have
\begin{align*}
I_{12}&\leq  cr^{\al p-n-1}\int_0^{(2 r)^2} \int^\infty_{t^{\frac{1}{2}}}  \int_{|x'| \leq 2r} t^{\frac{\beta}{2}}
                          x_n^{ -p} t^{\frac{p}{2}} dx' dx_n dt\\
 &\leq  cr^{\al p-2}\int_0^{(2 r)^2}t^{\frac{\beta+1}{2}}dt=cr^{\al p+\beta+1}\leq cT^{\frac{p\al +1+\beta}{2}}.
\end{align*}

From the estimate $\eqref{integral1}_2$ and  \eqref{077}, we have
\begin{align*}
I_{11}&\leq  cr^{\al p-n-1}\int_0^{(2 r)^2} \int_0^{t^{\frac{1}{2}}}  \int_{|x'| \leq 2r} x_n^{\beta}
                          \ln^p ( 1 +{ \frac{t}{x_n^2} })   dx' dx_n dt\\
  &\leq  c(\ep) r^{\al p-2}\int_0^{(2 r)^2} \int_0^{t^{\frac{1}{2}}} x_n^{\beta }   (\frac{t}{x_n^2})^{\ep p}
                          dx_n dt\\
 & \leq  c(\ep)r^{\al p-2}\int_0^{(2 r)^2}t^{\frac{\beta+1}{2}}dt=c(\ep)r^{\al p+\beta+1}\leq c(\ep)T^{\frac{p\al +1+\beta}{2}}.
\end{align*}

$\bullet$ For the estimate of $I_3$, we divide the domain of integration so that  $$
I_3=I_{31}+I_{32},$$
where
\begin{align*}
I_{31}= \int_0^{(2 r)^2} \int_0^{t^{\frac{1}{2}}} \int_{|x'| \geq 2r} x_n^{\beta}
                           |  u|^p dx' dx_n dt\\
                           I_{32}= \int_0^{(2 r)^2} \int^\infty_{t^{\frac{1}{2}}} \int_{|x'| \geq 2r} t^{ \frac12 \beta}
                           | u|^p dx' dx_n dt.
                                                      \end{align*}
From the estimate  \eqref{integral4} we have
\begin{align*}
I_{32}
&\leq  cr^{\al p-n-1+(n-1)p}\int_0^{(2 r)^2} \int^\infty_{t^{\frac{1}{2}}}  \int_{|x'| \geq 2r}t^{\frac{\beta}{2}} (|x'|^2+x_n^2)^{-\frac{np}{2}}t^{\frac{p}{2}} dx' dx_n dt.
                          \end{align*}
Taking $\ga: = -\frac{np}2 < -\frac{n-1}2$ and $A: =  x_n$
           in \eqref{0721-4},  we have
\begin{align*}
I_{32}
&\leq  cr^{\al p-n-1+(n-1)p}\int_0^{(2 r)^2} \int^\infty_{t^{\frac{1}{2}}}
 t^{\frac{\beta}{2}} (x_n^2+r^2)^{-\frac{np}{2}+\frac{n-1}{2}}t^{\frac{p}{2}}  dx_n dt\\
& \leq  cr^{\al p-1-p}\int_0^{(2 r)^2}t^{\frac{\beta +p}{2}}dt=
cr^{p\al +1+\beta}
\leq c T^{\frac{p\al +1+\beta}{2}}.
\end{align*}
From the estimate $\eqref{integral3}_2$ we have
\begin{align*}
I_{31}&\leq  cr^{\al p-n-1+(n-1)p}\int_0^{(2 r)^2} \int_0^{t^{\frac{1}{2}}} \int_{|x'| \geq 2r} x_n^{\beta+p}
                          (|x'|+x_n)^{-np} dx' dx_n dt.
                          \end{align*}
Taking $\ga: = -\frac{np}2 < -\frac{n-1}2$ and $A: =x_n  $
           in \eqref{0721-4},  we have
\begin{align*}
I_{31}&\leq  cr^{\al p-n-1+(n-1)p}\int_0^{(2 r)^2} \int_0^{t^{\frac{1}{2}}}x_n^{\beta+p}
                          (x_n+r)^{-np+n-1}  dx_n dt\\
&\leq  cr^{(\al-1) p-2}\int_0^{(2 r)^2}t^{\frac{\be+p+1}{2}} dt=cr^{\al p+\beta+1}\leq cT^{\frac{p\al +1+\beta}{2}}.
\end{align*}

\end{proof}

\section{Proof of Theorem \ref{b-1}}
\label{proofmain}
Theorem \ref{B-1} will be obtained by combining Theorem \ref{B-2} below and Proposition \ref{prop1}.
\begin{theo}\label{B-2}
Let $1< p< \infty$ and  $0<\alpha<1 $.
Let  $g=(g',0) \in {\mathcal B}^{\al,\frac12\al }_p ({\mathbb R}^{n-1}
\times (0,\infty))$.
Let $T>0$.
Then there is a solution $u$ of the Stokes system \eqref{maineq2} in ${\mathbb R}^{n-1}
\times (0,\infty)$ with boundary data $g$ such that
\begin{align}
\| u\|_{{\mathcal B}_p^{\al+\frac{1}{p},\frac12 \al+\frac{1}{2p}} ({\mathbb R}^n_+  \times (0,T))}\leq c(T) \| g\|_{{\mathcal B}^{\al, \frac\al2 }_p({\mathbb R}^{n-1}\times (0,T)) }.
\end{align}
\end{theo}

\begin{proof}
Let $T>0$ and $1<p<\infty$.
Let  $g \in {\mathcal B}_p^{\al,\frac12 \al} ({\mathbb R}^{n-1} \times (0,T)).$
 Let  $0 < \al <1-\frac{1}{p}$. From Lemma \ref{trace}   we have the following  estimate.
\begin{align*}
\| u\|^p_{{\mathcal B}_p^{\al+\frac{1}{p},\frac12 \al+\frac{1}{2p}} ({\mathbb R}^n_+ \times (0,T))}
  &\leq \int_0^T \int_{{\mathbb R}^n_+}
    (x_n \wedge t^\frac12)^{p-p\al-1} |D_{x}  u  |^p  +   (x_n \wedge t^\frac12)^{2p-p\al-1}| D_t u  |^p \\
     &\quad+(x_n \wedge t^\frac12)^{p-p\al-1}|u   |^p+(x_n \wedge t^\frac12)^{2p-p\al-1} |u|^pd xdt.
\end{align*}
By Theorem \ref{Rn}, we have
$$
\int_0^T \int_{{\mathbb R}^n_+}     (x_n \wedge t^\frac12)^{p-p\al-1} |D_{x}  u  |^p
+   (x_n \wedge t^\frac12)^{2p-p\al-1}| D_t u  |^p\leq c\|g\|^p_{{\mathcal B}_p^{\al,\frac12 \al} ({\mathbb R}^{n-1} \times (0,T))}.
    $$
By Theorem \ref{Rn4}, we have
$$
\int_0^T \int_{{\mathbb R}^n_+}(x_n \wedge t^\frac12)^{p-p\al-1}|u   |^p+(x_n \wedge t^\frac12)^{2p-p\al-1} |u|^pd xdt\leq c(T^{\frac{p}{2}}+T^{p})\|g\|^p_{{\mathcal B}_p^{\al,\frac12 \al} ({\mathbb R}^{n-1} \times (0,T))}.
$$
Therefore, we conclude that for $0 < \al <1-\frac{1}{p}$
\begin{align*}
\| u\|_{{\mathcal B}_p^{\al+\frac{1}{p},\frac12 \al+\frac{1}{2p}} ({\mathbb R}^n_+ \times (0,T))}
  \leq C(T)\|g\|_{{\mathcal B}_p^{\al,\frac12 \al} ({\mathbb R}^{n-1} \times (0,T))}.
  \end{align*}
Let  $1-\frac{1}{p} < \al < 1.$ From Lemma \ref{trace2}, we have the following estimate.
\begin{align*}
&\| u\|^p_{{\mathcal B}_p^{\al+\frac{1}{p},\frac12 \al+\frac{1}{2p}} ({\mathbb R}^n_+ \times (0,T))}\\
 & \leq c \int_0^T  \int_{{\mathbb R}^n_+}
  (x_n \wedge t^\frac12)^{2p-p\al-1}  \Big( |D^2_{x}  u  |^p  + |D_tu|^p)
  +  (x_n \wedge t^\frac12)^{p} |D_x D_t u  |^p\Big)\\
  &\quad+ (x_n \wedge t^\frac12)^{3p-p\al-1}|D_tu|^p+(x_n \wedge t^\frac12)^{2p-p\al-1}  |D_x u|^p\\
  & \quad +(x_n \wedge t^\frac12)^{3p-p\al-1}|D_xu|^p  + (x_n \wedge t^\frac12)^{2p-p\al-1} |u   |^p d x dt.
\end{align*}
By Theorem \ref{Rn}, we have
$$\int_0^T  \int_{{\mathbb R}^n_+}
  (x_n \wedge t^\frac12)^{2p-p\al-1}  \Big( |D^2_{x}  u  |^p  + |D_tu|^p+(x_n \wedge t^\frac12)^{p}|D_xD_tu|^p )\leq c\|g\|^p_{{\mathcal B}_p^{\al,\frac12 \al} ({\mathbb R}^{n-1} \times (0,T))}.
  $$
 Note that \begin{align*}
   &\int_0^T  \int_{{\mathbb R}^n_+} (x_n \wedge t^\frac12)^{3p-p\al-1}|D_tu|^p+(x_n \wedge t^\frac12)^{2p-p\al-1}  |D_x u|^p+(x_n \wedge t^\frac12)^{3p-p\al-1}|D_xu|^pdxdt\\
   &\leq \int_0^T  \int_{{\mathbb R}^n_+} T^p(x_n \wedge t^\frac12)^{2p-p\al-1}|D_tu|^p+(T^p+T^{2p})(x_n \wedge t^\frac12)^{p-p\al-1}  |D_x u|^pdxdt.
  \end{align*}
 Hence by Theorem \ref{Rn}, we have
 \begin{align*}
  & \int_0^T  \int_{{\mathbb R}^n_+} (x_n \wedge t^\frac12)^{3p-p\al-1}|D_tu|^p+(x_n \wedge t^\frac12)^{2p-p\al-1}  |D_x u|^p+(x_n \wedge t^\frac12)^{3p-p\al-1}|D_xu|^pdxdt\\
   &\leq c(T^p+T^{2p})\|g\|^p_{{\mathcal B}_p^{\al,\frac12 \al} ({\mathbb R}^{n-1} \times (0,T))}.\end{align*}
   From Theorem \ref{Rn4}, we have
   \begin{align*}
   \int_0^T  \int_{{\mathbb R}^n_+}(x_n \wedge t^\frac12)^{2p-p\al-1} |u   |^p d x dt\leq CT^{2p}\|g\|^p_{{\mathcal B}_p^{\al,\frac12 \al} ({\mathbb R}^{n-1} \times (0,T))}.
   \end{align*}
  Therefore, we conclude that for $1-\frac{1}{p}<\al<1$
\begin{align*}
\| u\|_{{\mathcal B}_p^{\al+\frac{1}{p},\frac12 \al+\frac{1}{2p}} ({\mathbb R}^n_+ \times (0,T))}
  \leq C(T)\|g\|_{{\mathcal B}_p^{\al,\frac12 \al} ({\mathbb R}^{n-1} \times (0,T))}.
  \end{align*}
 For $\al = 1 -\frac1p$, we use the real interpolation.
\end{proof}

\end{document}